\documentclass[12pt]{amsart}
\usepackage{amssymb}
\usepackage[all]{xy}

\usepackage{url}
\usepackage{verbatim}
\usepackage{comment}
\usepackage{geometry}
\usepackage{mathtools}
\usepackage{rotating}
\usepackage{tikz,graphicx}
\tikzstyle{vertex}=[circle, draw, inner sep=0pt, minimum size=4pt]
\newcommand{\vertex}{\node[vertex]}

\definecolor{Cerulean}{cmyk}{0.94,0.11,0,0}
\definecolor{ForestGreen}{cmyk}{0.91,0,0.88,0.12}
\definecolor{RedViolet}{cmyk}{0.07,0.90,0,0.34}

\newtheorem{theorem}{Theorem}[section]
\newtheorem{lemma}[theorem]{Lemma}
\newtheorem{proposition}[theorem]{Proposition}
\newtheorem{corollary}[theorem]{Corollary}
\theoremstyle{definition}
\newtheorem{definition}[theorem]{Definition}
\newtheorem{example}[theorem]{Example}

\newtheorem{remark}[theorem]{Remark}

\newcommand\bc{\mathbf{c}}

\newcommand\be{\mathbf{e}}

\newcommand\bx{\mathbf{x}}

\newcommand\Z{\mathbb{Z}}
\newcommand\R{\mathbb{R}}

\newcommand\calE{\mathcal{E}}
\newcommand\calF{\mathcal{F}}

\newcommand\calO{\mathcal{O}}
\newcommand\calP{\mathcal{P}}

\newcommand\calR{\mathcal{R}}

\newcommand\In{\mathrm{In}}
\newcommand\Out{\mathrm{Out}}
\newcommand\cone{\mathrm{cone}}

\newcommand\indeg{\mathrm{indeg}}
\newcommand\outdeg{\mathrm{outdeg}}
\newcommand\Adj{\mathrm{Adj}}

\newcommand{\DKK}{\mathrm{DKK}}
\newcommand{\flow}{f}

\DeclareMathOperator{\car}{car}

\title{Triangulations of Flow Polytopes, Ample Framings, and Gentle Algebras}

\author[von Bell]{Matias von Bell}
\address{Department of Mathematics\\
         University of Kentucky\\
\url{https://www.matiasvonbell.com/}}
\email{matias.vonbell@uky.edu}

\author[Braun]{Benjamin Braun}
\address{Department of Mathematics\\
         University of Kentucky\\
\url{https://sites.google.com/view/braunmath/}}
\email{benjamin.braun@uky.edu}

\author{Kaitlin Bruegge}
\address{Department of Mathematical Sciences\\
         University of Cincinnati\\
\url{https://researchdirectory.uc.edu/p/brueggkn}}
\email{kaitlin.bruegge@uc.edu}

\author[Hanely]{Derek Hanely}
\address{Department of Mathematics\\
         Penn State Behrend\\
\url{https://behrend.psu.edu/person/derek-hanely}}
\email{derek.hanely@psu.edu}

\author[Peterson]{Zachery Peterson}
\address{Department of Mathematics\\
         University of Kentucky\\
\url{https://math.as.uky.edu/users/ztpe226}}
\email{zachery.peterson@uky.edu}

\author[Serhiyenko]{Khrystyna Serhiyenko}
\address{Department of Mathematics\\
         University of Kentucky\\
\url{https://math.as.uky.edu/users/kse246}}
\email{khrystyna.serhiyenko@uky.edu}

\author[Yip]{Martha Yip}
\address{Department of Mathematics\\
         University of Kentucky\\
\url{https://www.ms.uky.edu/~myip/}}
\email{martha.yip@uky.edu}

\date{11 January 2024}

\begin{document}

\begin{abstract}
The cone of nonnegative flows for a directed acyclic graph (DAG) is known to admit regular unimodular triangulations induced by framings of the DAG.
These triangulations restrict to triangulations of the flow polytope for strength one flows, which are called DKK triangulations.
For a special class of framings called ample framings, these triangulations of the flow cone project to a complete fan.
We characterize the DAGs that admit ample framings, and we enumerate the number of ample framings for a fixed DAG.
We establish a connection between maximal simplices in DKK triangulations and $\tau$-tilting posets for certain gentle algebras, which allows us to impose a poset structure on the dual graph of any DKK triangulation for an amply framed DAG.
Using this connection, we are able to prove that for full DAGs, i.e., those DAGs with inner vertices having in-degree and out-degree equal to two, the flow polytopes are Gorenstein and have unimodal Ehrhart $h^\ast$-polynomials.
\end{abstract}

\thanks{The authors thank Emily Barnard, Christian Haase, Ford McElroy, Andreas Paffenholz, and Benjamin Nill for helpful comments.
BB, DH, and KB were partially supported by National Science Foundation award DMS-1953785.
KS was partially supported by National Science Foundation award DMS-2054255.
MY was partially supported
by Simons Collaboration Grant 429920.}

\maketitle 

\section{Introduction}\label{sec:intro}

Many problems in graph theory naturally translate to the setting of polytopes.
A prominent example of this is the study of flows on graphs and flow polytopes associated to transportation networks, which have been the subject of intense study in recent years. 
Given a directed acyclic graph (DAG) $G$ with capacity one on every edge, the polytope of flows of strength one is a lattice polytope with vertices corresponding to maximal routes in $G$.
In this paper, we call this the \emph{flow polytope for $G$}.
Flow polytopes are a central object of study in combinatorial optimization, and they also have important connections with various areas including representation theory~\cite{BV08}, diagonal harmonics~\cite{LMM16}, Grothendieck polynomials~\cite{LMS19, MS17}, and toric geometry~\cite{EM16}.

The cone of nonnegative flows for a DAG is known to admit regular unimodular triangulations induced by combinatorial structures called framings of the DAG.
These triangulations restrict to triangulations of the flow polytope, which are called \emph{DKK triangulations}, as they were initially studied by Danilov, Karzanov, and Koshevoy~\cite{DKK12}.
For a special class of framings called \emph{ample framings}, these triangulations of the flow cone project along a special simplex to a complete fan.
For each ample framing of $G$, we obtain a regular unimodular triangulation of this type.

Recently, the class of flow polytopes of $\nu$-caracol graphs $\mathrm{car}(\nu)$ were studied. 
By applying the Lidskii volume formula~\cite{BV08}, it was shown in~\cite{MM19} that the normalized volume of the flow polytope for $\mathrm{car}(\nu)$ is the generalized Catalan number which enumerates the number of lattice paths lying above a fixed path $\nu$.  
This motivated the study of unimodular triangulations of these polytopes.
It was shown in~\cite{BGMY} that the flow polytope for $\mathrm{car}(\nu)$ possesses DKK triangulations whose dual graphs are two ubiquitous lattice structures: the $\nu$-Tamari lattice and the lattice of order filters of a certain subset of the type $A$ root poset.
There are other families of flow polytopes whose normalized volumes are combinatorially interesting.
This leads to the question, do these flow polytopes also have unimodular triangulations that admit lattice structures?
To address this question, we establish a relationship between certain gentle algebras and triangulations of flow polytopes. In particular, we show that the dual graph of certain DKK triangulations is the Hasse diagram of the $\tau$-tilting poset for an associated gentle algebra.
Since finite $\tau$-tilting posets are lattices \cite[Corollary 3.12]{DIRRT}, we obtain the lattice structure on the dual graph of the DKK triangulation.  

Gentle algebras are an important class of finite dimensional algebras introduced in \cite{AssSko}, and their module categories are well understood in combinatorial terms by the work of \cite{ButlerRingel}.
However, in recent years the interest in gentle algebras significantly increased, and there has been a lot of new developments in this area.
In particular, their derived categories appear in the context of homological mirror symmetry \cite{HKK17, LP20}, and they can be modeled combinatorially via surfaces with marked points \cite{OpperPS}.
Moreover, the $\tau$-tilting posets of gentle algebras are related to the study of non-kissing complexes and non-crossing partitions, see \cite{PPP} and references therein.

Our focus is on flow polytopes for DAGs that admit ample framings.
This work should be of interest to combinatorialists and discrete geometers interested in flow polytopes, and to researchers in cluster algebras and representation theory.
We offer three main contributions:
\begin{enumerate}
    \item A classification of DAGs that admit ample framings (Lemma~\ref{lem:amplemeansfull}, Theorem~\ref{thm:ampleclassification}, Corollary~\ref{cor:amplecharacterizationiff}), and the enumeration of ample framings for such a DAG (Theorem~\ref{thm:poweroftwo}, Corollary~\ref{cor:poweroftwovalid}).
    \item A proof that the flow polytopes for a large class of DAGs, which we call \emph{full}, are Gorenstein (Theorem~\ref{thm:fullgorenstein}) and have unimodal $h^\ast$-polynomials (Corollary~\ref{cor:fullhstarunimodal}).
    \item A new connection between the dual graphs of DKK triangulations and $\tau$-tilting posets for gentle algebras (Theorem~\ref{routes-modules}, Theorem~\ref{maximal_cliques}); it is this connection that we use to prove the Gorenstein and unimodality results.
\end{enumerate}

To these ends, in Section~\ref{sec:ample}, we classify the DAGs that admit ample framings, and determine the size of the set of exceptional routes that form the special simplex for projection.
This classification identifies a particular class of DAGs called \emph{full} DAGs that play a key role in the study of ample framings.
In Section~\ref{sec:enumerating}, we enumerate the ample framings in DAGs and compute the number of ample framings for a particular family of examples.

In Section~\ref{sec:bijection}, for a full DAG $G$ with a DKK triangulation coming from an ample framing, we establish a connection to $\tau$-tilting posets for gentle algebras.
We prove that there is a bijection between the dual graph of the DKK triangulation of the flow polytope for $G$ and the $\tau$-tilting poset for a particular gentle algebra associated to $G$; this allows us to impose that poset structure on the dual graph of the DKK triangulation.
In Section~\ref{sec:gorenstein}, we use known properties of these $\tau$-tilting posets to show that the Ehrhart $h^\ast$-polynomial of the flow polytope for $G$ has symmetric coefficients.
Using this, we conclude that the flow polytope for a full DAG $G$ is always Gorenstein.
This combined with the regular unimodular triangulation allows us to conclude that its $h^\ast$-vector is unimodal, generalizing results of~\cite{AJR}.
Finally, we identify the routes in $G$ that yield special simplices arising from DKK triangulations.



\section{Background on Flow Polytopes}\label{sec:background}

\subsection{Flows and Flow Polytopes}\label{subsec:flowsandflowpolytopes}
In this section, we review fundamental background regarding flow polytopes; we generally follow the exposition by Danilov, Karzanov and Koshevoy in~\cite{DKK12}.
Let $G = (V,E)$ be a finite directed acyclic graph (DAG) with vertex set $V$ and edge set $E$. 
For each $v\in V$, let $\mathrm{in}(v)$ and $\mathrm{out}(v)$ denote the incoming and outgoing edges of $v$ respectively. 
A vertex $v$ is called a \emph{source} if $\mathrm{in}(v) =\emptyset$ and it is called a \emph{sink} if $\mathrm{out}(v) = \emptyset$. Any other vertices are called \emph{inner vertices}. 
A \emph{route} in $G$ is a maximal path in $G$, i.e., a path beginning at a source and ending at a sink. 
The set of all routes is denoted $\calP=\calP(G)$. 
If $v$ is a vertex on the route $R$ then it splits $R$ into two subpaths. Let $Rv$ denote the subpath of $R$ from the source of $R$ to $v$, and let $vR$ denote the subpath from $v$ to the sink. If the edges in $G$ are assigned some linear order $e_1$, $e_2$, $\ldots$, $e_{|E|}$, then for a route $R \in \calP(G)$ we define its characteristic vector to be $v_R =\sum_{e_i\in R} \be_{i}$.
If the edges in $G$ are not assigned a linear order, then we will index the space $\R^{|E|}$ by the edges $e\in E$.

We will use the following graph in our running example. 


\begin{example}
The graph $G=\car(8)$ in Figure~\ref{fig:car8} is known as the caracol graph on eight vertices.
The source and sink are vertices $1$ and $8$ respectively. The vertices in $\{2,3,\ldots,7\}$ are inner vertices. 
The edges $(1,3)$, $(3,4)$, and $(4,8)$ form a route $R$. 
In this case $R4$ is the path consisting of edges $(1,3)$ and $(3,4)$ while $4R$ is the single edge $(4,8)$.  
If the edges of $R$ appear as the second, eighth, and fourteenth edges in a linear ordering of the edges of $G$, then the characteristic vector $v_R$ is $(0,1, 0,0, 0,0, 0,1, 0,0, 0,0, 0,1, 0,0, 0)$. 
\end{example}

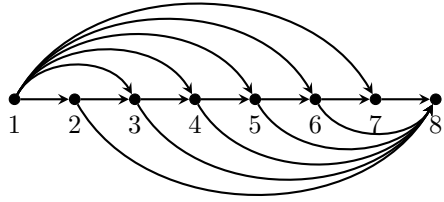
\begin{figure}
\begin{center}
\begin{tikzpicture}[scale=1]
\begin{scope}[scale=0.8, xshift=0, yshift=0]
	\vertex[fill,label=below:\footnotesize{$1$}](a0) at (0,0) {};
	\vertex[fill,label=below:\footnotesize{$2$}](a1) at (1,0) {};
	\vertex[fill,label=below:\footnotesize{$3$}](a2) at (2,0) {};
	\vertex[fill,label=below:\footnotesize{$4$}](a3) at (3,0) {};
	\vertex[fill,label=below:\footnotesize{$5$}](a4) at (4,0) {};
	\vertex[fill,label=below:\footnotesize{$6$}](a5) at (5,0) {};
	\vertex[fill,label=below:\footnotesize{$7$}](a6) at (6,0) {};
	\vertex[fill,label=below:\footnotesize{$8$}](a7) at (7,0) {};

    \draw[-stealth, thick] (a0) -- (a1);
    \draw[-stealth, thick] (a0) to[out=60,in=120] (a2);
    \draw[-stealth, thick] (a0) to[out=60,in=120] (a3);
    \draw[-stealth, thick] (a0) to[out=60,in=120] (a4);
    \draw[-stealth, thick] (a0) to[out=60,in=120] (a5);
    \draw[-stealth, thick] (a0) to[out=60,in=120] (a6);

	\draw[-stealth, thick] (a1) to[out=-60,in=-120] (a7);
	\draw[-stealth, thick] (a2) to[out=-60,in=-120] (a7);
	\draw[-stealth, thick] (a3) to[out=-60,in=-120] (a7);
	\draw[-stealth, thick] (a4) to[out=-60,in=-120] (a7);
	\draw[-stealth, thick] (a5) to[out=-60,in=-120] (a7);
    \draw[-stealth, thick] (a6) -- (a7);
    
	\draw[-stealth, thick] (a1)--(a2);
	\draw[-stealth, thick] (a2)--(a3);
    \draw[-stealth, thick] (a3)--(a4);
    \draw[-stealth, thick] (a4)--(a5);
    \draw[-stealth, thick] (a5)--(a6);

\end{scope}
\end{tikzpicture}
\end{center}
\caption{The graph $\car(8)$.} 
\label{fig:car8}
\end{figure}

\begin{definition}\label{def:flow}
A \emph{flow} $\flow$ on a DAG $G$ is a function $\flow: E \to \R$ which preserves flow at each inner vertex, i.e., for every inner vertex $v$ we have
\[
\sum_{e\,\in \, \mathrm{in}(v)} \flow(e) = \sum_{e\,\in \, \mathrm{out}(v)} \flow(e) \, .
\]
Let $\calF = \calF(G)$ denote the space of flows on $G$, and let $\calF_+ = \calF_+(G)$ denote the cone of flows satisfying $f(e)\geq 0$ for all edges $e\in E$.
The \emph{flow polytope} $\calF_1 = \calF_1(G)$ is the set of all nonnegative flows on $G$ of size one, i.e., flows satisfying
\[
\sum_{\substack{v \text{ is a source} \\ e \,\in \, \mathrm{out}(v)}} \flow(e) = 1 \, .
\]
\end{definition}

The following proposition is a straightforward consequence of the total unimodularity of the signed incidence matrix for $G$~\cite[Theorem 4.9]{CCZ} and row reduction.

\begin{proposition}\label{prop:flowbasics}
Given a DAG $G$, the set of flows $\calF(G)$ forms a vector subspace of $\R^{|E|}$ spanned by the characteristic vectors of the routes and has dimension
\[
\dim(\calF(G)) = |E| - \#\{v \in V(G) : v \text{ is an inner vertex }\} \, .
\]
The vertices of $\calF_1$ are the characteristic vectors $\{v_R : R\in \calP(G)\}$ and 
\[
\dim( \mathcal{F}_1) = |E| - \#\{v \in V(G) : v \text{ is an inner vertex } \} -1 \, .
\]
\end{proposition}

In some cases, we can simplify $G$ by contracting some of its edges without changing the lattice-polyhedral structure of $\calF_+(G)$ or $\calF_1$.
Note that all of the facets of $\calF_+(G)$ are given by $x_e= 0$ for an edge $e\in E$.
An edge is said to be \emph{idle} if it is the only incoming or outgoing edge from an inner vertex. 
Contracting an idle edge $e$ corresponds geometrically to a projection along the coordinate $x_e$.
We can find non-redundant facet descriptions of the cone of flows via sequences of idle edge contractions in the following manner.

\begin{definition}\label{def:reductiontofull}
Given a DAG $G$, produce a new DAG $G^1$ by contracting an idle edge in $G$.
Inductively construct $G^i$ by contracting an idle edge in $G^{i-1}$, and continue this process until there are no idle edges, resulting in the DAG $H$.
We call $H$ a \emph{complete contraction} of $G$.
\end{definition}

An example is given in Figure~\ref{fig:nu-graph}.

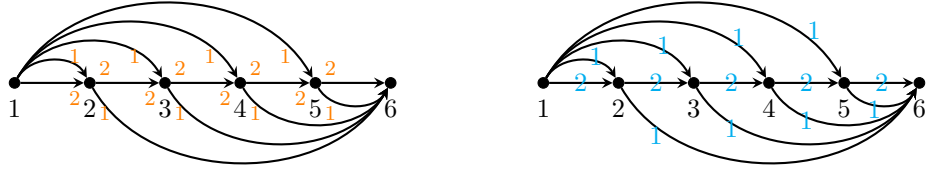
\begin{figure}
\begin{center}
\begin{tikzpicture}
\begin{scope}[scale=1, xshift=0, yshift=0]
	\vertex[fill,label=below:\footnotesize{$1$}](a1) at (1,0) {};
	\vertex[fill,label=below:\footnotesize{$2$}](a2) at (2,0) {};
	\vertex[fill,label=below:\footnotesize{$3$}](a3) at (3,0) {};
	\vertex[fill,label=below:\footnotesize{$4$}](a4) at (4,0) {};
	\vertex[fill,label=below:\footnotesize{$5$}](a5) at (5,0) {};
	\vertex[fill,label=below:\footnotesize{$6$}](a6) at (6,0) {};
	
	\draw[-stealth, thick] (a1)--(a2);
	\draw[-stealth, thick] (a2)--(a3);
	\draw[-stealth, thick] (a3)--(a4);
	\draw[-stealth, thick] (a4)--(a5);
	\draw[-stealth, thick] (a5)--(a6);
	\draw[-stealth, thick] (a1) to[out=60,in=120] (a2);
	\draw[-stealth, thick] (a1) to[out=60,in=120] (a3);
	\draw[-stealth, thick] (a1) to[out=60,in=120] (a4);
	\draw[-stealth, thick] (a1) to[out=60,in=120] (a5);
	\draw[-stealth, thick] (a1) to[out=60,in=120] (a6);
	\draw[-stealth, thick] (a1) to[out=-60,in=-120] (a6);
	\draw[-stealth, thick] (a2) to[out=-60,in=-120] (a6);
	\draw[-stealth, thick] (a3) to[out=-60,in=-120] (a6);
	\draw[-stealth, thick] (a4) to[out=-60,in=-120] (a6);
	\draw[-stealth, thick] (a5) to[out=-60,in=-120] (a6);
	
    \node[] at (1.8,0.35) {\textcolor{orange}{\tiny$1$}};
    \node[] at (1.8,-0.2) {\textcolor{orange}{\tiny$2$}};
    \node[] at (2.2,-0.4) {\textcolor{orange}{\tiny$1$}};
    \node[] at (2.2,0.2) {\textcolor{orange}{\tiny$2$}};
    \node[] at (2.6,0.35) {\textcolor{orange}{\tiny$1$}};
    \node[] at (2.8,-0.2) {\textcolor{orange}{\tiny$2$}};
    \node[] at (3.2,-0.4) {\textcolor{orange}{\tiny$1$}};
    \node[] at (3.2,0.2) {\textcolor{orange}{\tiny$2$}};
    \node[] at (3.6,0.35) {\textcolor{orange}{\tiny$1$}};
    \node[] at (3.8,-0.2) {\textcolor{orange}{\tiny$2$}};
    \node[] at (4.2,-0.4) {\textcolor{orange}{\tiny$1$}};
    \node[] at (4.2,0.2) {\textcolor{orange}{\tiny$2$}};
    \node[] at (4.6,0.35) {\textcolor{orange}{\tiny$1$}};
    \node[] at (4.8,-0.2) {\textcolor{orange}{\tiny$2$}};
    \node[] at (5.2,-0.4) {\textcolor{orange}{\tiny$1$}};
    \node[] at (5.2,0.2) {\textcolor{orange}{\tiny$2$}};
\end{scope}

\begin{scope}[scale=1, xshift=200, yshift=0]
	\vertex[fill,label=below:\footnotesize{$1$}](a1) at (1,0) {};
	\vertex[fill,label=below:\footnotesize{$2$}](a2) at (2,0) {};
	\vertex[fill,label=below:\footnotesize{$3$}](a3) at (3,0) {};
	\vertex[fill,label=below:\footnotesize{$4$}](a4) at (4,0) {};
	\vertex[fill,label=below:\footnotesize{$5$}](a5) at (5,0) {};
	\vertex[fill,label=below:\footnotesize{$6$}](a6) at (6,0) {};
	
	\draw[-stealth, thick] (a1)--(a2);
	\draw[-stealth, thick] (a2)--(a3);
	\draw[-stealth, thick] (a3)--(a4);
	\draw[-stealth, thick] (a4)--(a5);
	\draw[-stealth, thick] (a5)--(a6);
	\draw[-stealth, thick] (a1) to[out=60,in=120] (a2);
	\draw[-stealth, thick] (a1) to[out=60,in=120] (a3);
	\draw[-stealth, thick] (a1) to[out=60,in=120] (a4);
	\draw[-stealth, thick] (a1) to[out=60,in=120] (a5);
	\draw[-stealth, thick] (a1) to[out=60,in=120] (a6);
	\draw[-stealth, thick] (a1) to[out=-60,in=-120] (a6);
	\draw[-stealth, thick] (a2) to[out=-60,in=-120] (a6);
	\draw[-stealth, thick] (a3) to[out=-60,in=-120] (a6);
	\draw[-stealth, thick] (a4) to[out=-60,in=-120] (a6);
	\draw[-stealth, thick] (a5) to[out=-60,in=-120] (a6);
	
    \node[] at (5.5,0) {\textcolor{cyan}{\footnotesize$2$}};
    \node[] at (4.5,0) {\textcolor{cyan}{\footnotesize$2$}};	
    \node[] at (3.5,0) {\textcolor{cyan}{\footnotesize$2$}};	
    \node[] at (2.5,0) {\textcolor{cyan}{\footnotesize$2$}};	
    \node[] at (1.5,0) {\textcolor{cyan}{\footnotesize$2$}};	

    \node[] at (1.7,0.35) {\textcolor{cyan}{\footnotesize$1$}};
    \node[] at (2.6,0.5) {\textcolor{cyan}{\footnotesize$1$}};
    \node[] at (3.6,0.6) {\textcolor{cyan}{\footnotesize$1$}};
    \node[] at (4.6,0.7) {\textcolor{cyan}{\footnotesize$1$}};    

    \node[] at (2.5,-0.7) {\textcolor{cyan}{\footnotesize$1$}};
    \node[] at (3.5,-0.6) {\textcolor{cyan}{\footnotesize$1$}};
    \node[] at (4.5,-0.5) {\textcolor{cyan}{\footnotesize$1$}};
    \node[] at (5.4,-0.35) {\textcolor{cyan}{\footnotesize$1$}};
\end{scope}
\end{tikzpicture}
\end{center}
\caption{Example of a complete contraction of the graph $\car(8)$ where the first and last idle edges have been contracted. The left graph shows a framing at the inner vertices labeled in orange.
This framing induces a labeling on each edge $(u,v)$ representing its ordering in both \(\mathrm{in}(v)\) and \(\mathrm{out}(u)\). The induced edge-labeling is shown in blue on the right.
}
\label{fig:nu-graph}
\end{figure}

\begin{proposition}\label{prop:contractidle}
If $G$ is a DAG, then the set of facets of the cone $\calF_+(G)$ can be identified with the set of edges in a complete contraction of $G$.
\end{proposition}

\subsection{Framings, Coherent Routes, and the DKK Triangulation}\label{subsec:framingsroutesDKK}

In this subsection we review basic definitions and properties regarding framed graphs, routes, and coherence, including a main result of Danilov, Karzanov, and Koshevoy~\cite{DKK12} constructing triangulations of flow cones for framed graphs.

\begin{definition}\label{def:framing}
Let $G$ be a DAG.
For each inner vertex $v$ of $G$, assign a linear order to the edges in $\mathrm{in}(v)$ and also assign a linear order to the edges in $\mathrm{out}(v)$.
This assignment is called a \emph{framing} of $G$, which we denote by $F$.
We call a DAG $G$ with a framing $F$ a \emph{framed graph}, which we often denote by $[G,F]$.
If $e$ is less than $f$ in the linear order for $F$ on $\mathrm{in}(v)$, we write $e\prec_{F,\mathrm{in}(v)} f$ (and similarly for $\mathrm{out}(v)$).
When $F$ and/or $\mathrm{in}(v)$ or $\mathrm{out}(v)$ is clear, we sometimes drop one or both subscripts from $\prec_{F,\mathrm{in}(v)}$.
\end{definition}

\begin{example}\label{ex:car8framing}
Consider the DAG in Figure~\ref{fig:nu-graph}.
Define a framing by assigning the linear order $\mathrm{in}(v)=\{(j,v)<(i,v)\}$ when $j<i$, assigning the linear order $\mathrm{out}(v)=\{(v,j)<(v,i)\}$ when $i<j$, and by ordering the multiedges from $1$ to $2$ and from $5$ to $6$ by setting the shorter length edge in the picture to be second in the pair.
This is called the \emph{length framing} because the longer length edges in the picture are first in their orders and the shorter edges are second.
\end{example}

For a DAG $G$ and an inner vertex $v$, let $\In(v)$ denote the set of maximal paths in $G$ from a source to $v$ and let $\Out(v)$ denote the set of maximal paths in $G$ from $v$ to a sink.
Given a framing $F$ on $G$, the linear orders from $F$ induce an ordering on the sets $\In(v)$ and $\Out(v)$ as follows.

\begin{definition}\label{def:inoutpathorders}
Let $[G,F]$ be a framed graph.
Let $P$ and $Q$ be paths in $\Out(v)$ that coincide on the subpaths $P'\subset P$ and $Q' \subset Q$ that begin at $v$ and end at $w$.
Suppose that the vertices following $w$ on $P$ and $Q$ are distinct; call them $w_P$ and $w_Q$.
Set $P\prec_{F,\Out(v)} Q$ if $(w,w_P)\prec_{F,\mathrm{out}(w)} (w,w_Q)$, and similarly for $P\prec_{F,\In(v)} Q$.
When $F$ and/or $\In(v)$ or $\Out(v)$ is clear from context, we will sometimes drop one or both subscripts from $\prec_{F,\In(v)}$.
\end{definition}

\begin{example}\label{ex:car8inoutpaths}
Consider the paths $P=346$ and $Q=3456$ in the graph $H$ from Figure~\ref{fig:nu-graph}, using the length framing given in Example~\ref{ex:car8framing}.
In the set $\Out(3)$, we have that $P\prec_{F,\mathrm{Out}(3)} Q$.
For the paths $A=14$ and $B=1234$ in $\In(4)$, we have that $A\prec_{F,\In(4)} B$.
\end{example}

\begin{definition}\label{def:coherent}
Suppose that $P$ and $Q$ are routes in a framed graph $[G,F]$ that intersect at a common inner vertex $v$.
$P$ and $Q$ are \emph{in conflict}, also called \emph{conflicting}, if $Pv\prec_{\In(v)} Qv$ and $vQ \prec_{\Out(v)} vP$.
If $P$ and $Q$ are not conflicting at $v$, then they are \emph{coherent} at $v$.
$P$ and $Q$ are called \emph{coherent} if they are coherent at every inner vertex $v$ that is contained in both $P$ and $Q$.
\end{definition}

\begin{example}\label{ex:car8coherent}
Consider the routes $P=1346$ and $Q=1236$ in the graph $H$ from Figure~\ref{fig:nu-graph}, using the length framing given in Example~\ref{ex:car8framing}.
Then $3$ is an inner vertex common to both $P$ and $Q$, and $P$ and $Q$ are in conflict at $3$.
Hence, $P$ and $Q$ are not coherent.
As a second example, consider the routes $P'=13456$ and $Q'=12346$.
These intersect at both $3$ and $4$ and share the edge $(3,4)$.
The routes $P'$ and $Q'$ are also in conflict at both $3$ and $4$.
On the other hand, the pair of routes $136$ and $123456$ is coherent.
\end{example}

\begin{definition}\label{def:cliqueexceptional}
Given a framed graph $[G,F]$, a \emph{clique} is a set of pairwise-coherent routes in $G$.
If a route $R$ in $G$ is coherent with every other route in $G$, we say that $R$ is \emph{exceptional}.
In this case, $R$ is an element of every maximal clique of routes.
\end{definition}

\begin{example}\label{ex:cliqueandexceptional}
In the graph from Figure~\ref{fig:nu-graph}, using the length framing given in Example~\ref{ex:car8framing}, there are five exceptional routes: $123456$, $126$, $136$, $146$, and $156$, where the choice of which of the two edges labeled $12$ is made to ensure coherence with all other routes (and similarly for $56$).
An example of a maximal clique is given by these five routes together with $123456$, $13456$, $1456$, and $156$, where again we select the appropriate multiedge from $12$ and $56$ to avoid conflicts. The routes of this clique are illustrated in Figure~\ref{fig:exceptional routes of contracted car(8)}.
\end{example}

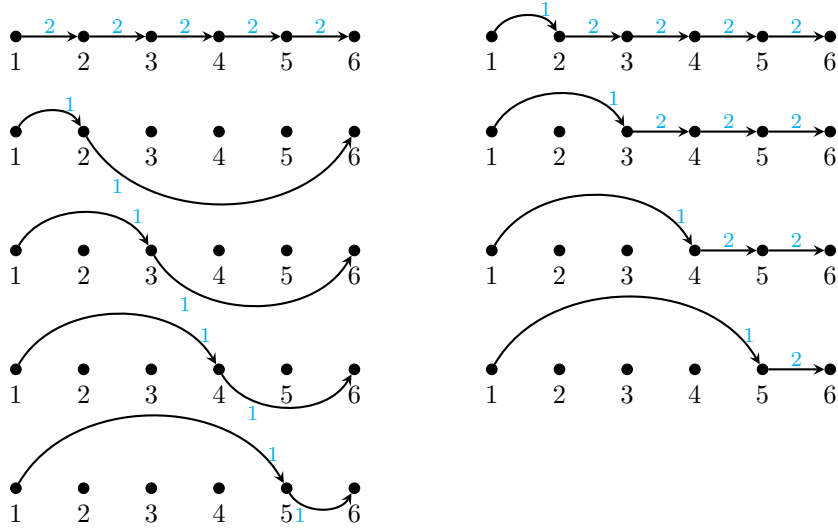
\begin{figure}
\begin{center}
\begin{tikzpicture}
\begin{scope}[scale=0.9, xshift=0, yshift=0]
	\vertex[fill,label=below:\footnotesize{$1$}](a1) at (1,0) {};
	\vertex[fill,label=below:\footnotesize{$2$}](a2) at (2,0) {};
	\vertex[fill,label=below:\footnotesize{$3$}](a3) at (3,0) {};
	\vertex[fill,label=below:\footnotesize{$4$}](a4) at (4,0) {};
	\vertex[fill,label=below:\footnotesize{$5$}](a5) at (5,0) {};
	\vertex[fill,label=below:\footnotesize{$6$}](a6) at (6,0) {};
	
	\draw[-stealth, thick] (a1)--(a2);
	\draw[-stealth, thick] (a2)--(a3);
	\draw[-stealth, thick] (a3)--(a4);
	\draw[-stealth, thick] (a4)--(a5);
	\draw[-stealth, thick] (a5)--(a6);

    
    \node[] at (1.5,0.15) {\textcolor{cyan}{\tiny$2$}}; 
    \node[] at (2.5,0.15) {\textcolor{cyan}{\tiny$2$}}; 
    \node[] at (3.5,0.15) {\textcolor{cyan}{\tiny$2$}}; 
    \node[] at (4.5,0.15) {\textcolor{cyan}{\tiny$2$}}; 
    \node[] at (5.5,0.15) {\textcolor{cyan}{\tiny$2$}}; 
    

\end{scope}

\begin{scope}[scale=0.9, xshift=0, yshift=-40]
	\vertex[fill,label=below:\footnotesize{$1$}](a1) at (1,0) {};
	\vertex[fill,label=below:\footnotesize{$2$}](a2) at (2,0) {};
	\vertex[fill,label=below:\footnotesize{$3$}](a3) at (3,0) {};
	\vertex[fill,label=below:\footnotesize{$4$}](a4) at (4,0) {};
	\vertex[fill,label=below:\footnotesize{$5$}](a5) at (5,0) {};
	\vertex[fill,label=below:\footnotesize{$6$}](a6) at (6,0) {};
	
	\draw[-stealth, thick] (a1) to[out=60,in=120] (a2);
	\draw[-stealth, thick] (a2) to[out=-60,in=-120] (a6);

    \node[] at (1.8,0.4) {\textcolor{cyan}{\tiny$1$}}; 
    
    
    \node[] at (2.5,-0.8) {\textcolor{cyan}{\tiny$1$}}; 
\end{scope}

\begin{scope}[scale=0.9, xshift=0, yshift=-90]
	\vertex[fill,label=below:\footnotesize{$1$}](a1) at (1,0) {};
	\vertex[fill,label=below:\footnotesize{$2$}](a2) at (2,0) {};
	\vertex[fill,label=below:\footnotesize{$3$}](a3) at (3,0) {};
	\vertex[fill,label=below:\footnotesize{$4$}](a4) at (4,0) {};
	\vertex[fill,label=below:\footnotesize{$5$}](a5) at (5,0) {};
	\vertex[fill,label=below:\footnotesize{$6$}](a6) at (6,0) {};
	
 	\draw[-stealth, thick] (a1) to[out=60,in=120] (a3);
 	\draw[-stealth, thick] (a3) to[out=-60,in=-120] (a6);

    \node[] at (2.8,0.5) {\textcolor{cyan}{\tiny$1$}}; 
    
    
    \node[] at (3.5,-0.8) {\textcolor{cyan}{\tiny$1$}}; 
\end{scope}

\begin{scope}[scale=0.9, xshift=0, yshift=-140]
	\vertex[fill,label=below:\footnotesize{$1$}](a1) at (1,0) {};
	\vertex[fill,label=below:\footnotesize{$2$}](a2) at (2,0) {};
	\vertex[fill,label=below:\footnotesize{$3$}](a3) at (3,0) {};
	\vertex[fill,label=below:\footnotesize{$4$}](a4) at (4,0) {};
	\vertex[fill,label=below:\footnotesize{$5$}](a5) at (5,0) {};
	\vertex[fill,label=below:\footnotesize{$6$}](a6) at (6,0) {};
	
 	\draw[-stealth, thick] (a1) to[out=60,in=120] (a4);
 	\draw[-stealth, thick] (a4) to[out=-60,in=-120] (a6);

     \node[] at (3.8,0.5) {\textcolor{cyan}{\tiny$1$}}; 
    
    
     \node[] at (4.5,-0.65) {\textcolor{cyan}{\tiny$1$}}; 

\end{scope}

\begin{scope}[scale=0.9, xshift=0, yshift=-190]
	\vertex[fill,label=below:\footnotesize{$1$}](a1) at (1,0) {};
	\vertex[fill,label=below:\footnotesize{$2$}](a2) at (2,0) {};
	\vertex[fill,label=below:\footnotesize{$3$}](a3) at (3,0) {};
	\vertex[fill,label=below:\footnotesize{$4$}](a4) at (4,0) {};
	\vertex[fill,label=below:\footnotesize{$5$}](a5) at (5,0) {};
	\vertex[fill,label=below:\footnotesize{$6$}](a6) at (6,0) {};
	
 	\draw[-stealth, thick] (a1) to[out=60,in=120] (a5);
 	\draw[-stealth, thick] (a5) to[out=-60,in=-120] (a6);

     \node[] at (4.8,0.5) {\textcolor{cyan}{\tiny$1$}}; 
    
    
     \node[] at (5.2,-0.4) {\textcolor{cyan}{\tiny$1$}}; 

\end{scope}


\begin{scope}[scale=0.9, xshift=200, yshift=0]
	\vertex[fill,label=below:\footnotesize{$1$}](a1) at (1,0) {};
	\vertex[fill,label=below:\footnotesize{$2$}](a2) at (2,0) {};
	\vertex[fill,label=below:\footnotesize{$3$}](a3) at (3,0) {};
	\vertex[fill,label=below:\footnotesize{$4$}](a4) at (4,0) {};
	\vertex[fill,label=below:\footnotesize{$5$}](a5) at (5,0) {};
	\vertex[fill,label=below:\footnotesize{$6$}](a6) at (6,0) {};
	
	\draw[-stealth, thick] (a2)--(a3);
	\draw[-stealth, thick] (a3)--(a4);
	\draw[-stealth, thick] (a4)--(a5);
	\draw[-stealth, thick] (a5)--(a6);
	\draw[-stealth, thick] (a1) to[out=60,in=120] (a2);

    \node[] at (1.8,0.4) {\textcolor{cyan}{\tiny$1$}}; 
    
    \node[] at (2.5,0.15) {\textcolor{cyan}{\tiny$2$}}; 
    \node[] at (3.5,0.15) {\textcolor{cyan}{\tiny$2$}}; 
    \node[] at (4.5,0.15) {\textcolor{cyan}{\tiny$2$}}; 
    \node[] at (5.5,0.15) {\textcolor{cyan}{\tiny$2$}}; 
    

\end{scope}

\begin{scope}[scale=0.9, xshift=200, yshift=-40]
	\vertex[fill,label=below:\footnotesize{$1$}](a1) at (1,0) {};
	\vertex[fill,label=below:\footnotesize{$2$}](a2) at (2,0) {};
	\vertex[fill,label=below:\footnotesize{$3$}](a3) at (3,0) {};
	\vertex[fill,label=below:\footnotesize{$4$}](a4) at (4,0) {};
	\vertex[fill,label=below:\footnotesize{$5$}](a5) at (5,0) {};
	\vertex[fill,label=below:\footnotesize{$6$}](a6) at (6,0) {};
	
	\draw[-stealth, thick] (a3)--(a4);
	\draw[-stealth, thick] (a4)--(a5);
	\draw[-stealth, thick] (a5)--(a6);
 	\draw[-stealth, thick] (a1) to[out=60,in=120] (a3);

    \node[] at (2.8,0.5) {\textcolor{cyan}{\tiny$1$}}; 
    
    \node[] at (3.5,0.15) {\textcolor{cyan}{\tiny$2$}}; 
    \node[] at (4.5,0.15) {\textcolor{cyan}{\tiny$2$}}; 
    \node[] at (5.5,0.15) {\textcolor{cyan}{\tiny$2$}}; 
    

\end{scope}

\begin{scope}[scale=0.9, xshift=200, yshift=-90]
	\vertex[fill,label=below:\footnotesize{$1$}](a1) at (1,0) {};
	\vertex[fill,label=below:\footnotesize{$2$}](a2) at (2,0) {};
	\vertex[fill,label=below:\footnotesize{$3$}](a3) at (3,0) {};
	\vertex[fill,label=below:\footnotesize{$4$}](a4) at (4,0) {};
	\vertex[fill,label=below:\footnotesize{$5$}](a5) at (5,0) {};
	\vertex[fill,label=below:\footnotesize{$6$}](a6) at (6,0) {};
	
	\draw[-stealth, thick] (a4)--(a5);
	\draw[-stealth, thick] (a5)--(a6);
 	\draw[-stealth, thick] (a1) to[out=60,in=120] (a4);

     \node[] at (3.8,0.5) {\textcolor{cyan}{\tiny$1$}}; 
    
    \node[] at (4.5,0.15) {\textcolor{cyan}{\tiny$2$}}; 
    \node[] at (5.5,0.15) {\textcolor{cyan}{\tiny$2$}}; 
    

\end{scope}

\begin{scope}[scale=0.9, xshift=200, yshift=-140]
	\vertex[fill,label=below:\footnotesize{$1$}](a1) at (1,0) {};
	\vertex[fill,label=below:\footnotesize{$2$}](a2) at (2,0) {};
	\vertex[fill,label=below:\footnotesize{$3$}](a3) at (3,0) {};
	\vertex[fill,label=below:\footnotesize{$4$}](a4) at (4,0) {};
	\vertex[fill,label=below:\footnotesize{$5$}](a5) at (5,0) {};
	\vertex[fill,label=below:\footnotesize{$6$}](a6) at (6,0) {};
	
	\draw[-stealth, thick] (a5)--(a6);
 	\draw[-stealth, thick] (a1) to[out=60,in=120] (a5);

     \node[] at (4.8,0.5) {\textcolor{cyan}{\tiny$1$}}; 
    
    \node[] at (5.5,0.15) {\textcolor{cyan}{\tiny$2$}}; 
    

\end{scope}

\end{tikzpicture}

\end{center}
\caption{Given the length framing of the complete contraction of the graph \(\car(8)\) described in Example~\ref{ex:car8framing}, the above routes form a maximal clique. The routes in the left column are the exceptional routes for this framing. }
\label{fig:exceptional routes of contracted car(8)}
\end{figure}

The following theorem~\cite[Theorems 1 and 2]{DKK12} shows that the cliques in $[G,F]$ have particularly nice geometric properties.

\begin{theorem}[Danilov, Karzanov, Koshevoy~\cite{DKK12}]\label{thm:DKKtriangulation}
Let $G$ be a DAG with framing $F$.
The set of cliques for $G$ with respect to $F$ forms a regular unimodular triangulation of $\calF_+(G)$ that restricts to a regular unimodular triangulation of $\calF_1(G)$.
\end{theorem}

\begin{definition}\label{def:DKKtriangulation}
The triangulation from Theorem~\ref{thm:DKKtriangulation} is called the \emph{DKK triangulation corresponding to F} and denote it $\DKK(G,F)$.
\end{definition}

Given a framed graph $[G,F]$, let $\calE$ denote the set of exceptional routes in $G$.
Since $\calE$ is contained in every facet of $\DKK(G,F)$, this implies that $\DKK(G,F)$ is obtained as the join of $\calE$ with the link of $\calE$ in $\DKK(G,F)$.
Hence, it is of interest to ask what happens to $\calF_+(G)$ when we quotient out by the linear span of $\calE$.

\begin{definition}\label{def:reduced}
Given a framed graph $[G,F]$ with exceptional set $\calE$, let $\calF(G)_{red}=\calF_{red}:=\calF(G)/\mathrm{span}_{\R}(\calE)$ be the \emph{reduced space} of $\calF$, and let $p:\calF\to\calF_{red}$ denote the projection map.
Let $\calF_{+,red}:=p(\calF_{+})$, which we call the \emph{reduced cone}.
Let $\DKK(G,F)_{red}$ denote the \emph{reduced fan} formed by the set of simplicial cones $\{p(C):C \in \DKK(G,F)\}$.
\end{definition}

A question of interest to Danilov, Karzanov, and Koshevoy, motivated by a conjecture of Petersen, Pylyavskyy, and Speyer~\cite{PPS}, is to determine when the fan $\DKK(G,F)_{red}$ is complete, meaning that the union of simplicial cones in $\DKK(G,F)_{red}$ is equal to $\calF_{red}$.
A combinatorial characterization of framings that yield complete reduced fans is the following.

\begin{definition}\label{def:ample}
Let $[G,F]$ be a framed graph with exceptional set $\calE$.
If $\calE$ is not contained in any facet of $\calF_{+}$, then we say both $\calE$ and $F$ are \emph{ample}.
\end{definition}

The following theorem~\cite[Proposition~5]{DKK12} gives both a geometric and a combinatorial characterization of ample framings.

\begin{theorem}[Danilov, Karzanov, Koshevoy~\cite{DKK12}]\label{thm:DKKample}
Given a framed graph $[G,F]$, the following conditions are equivalent:
\begin{enumerate}
    \item $F$ is ample,
    \item $\calF(G)_{+,red}=\calF_{red}$, i.e., $\calF(G)_{+,red}$ is a complete fan,
    \item each non-idle edge belongs to an exceptional route for $F$.
\end{enumerate}
\end{theorem}


\section{Ample Framings}\label{sec:ample}

In this section we investigate the class of DAGs that admit ample framings.
We begin by considering DAGs that do not have idle edges.

\subsection{Full Framed Graphs}

DAGs with an ample framing and no idle edges have a restricted combinatorial structure, as the following lemma demonstrates.

\begin{lemma}\label{lem:amplemeansfull}
For a DAG $G$ with ample framing $F$ and no idle edges, every inner vertex $v$ has $\indeg(v)=2=\outdeg(v)$.
\end{lemma}

\begin{proof}
Because there are no idle edges in $G$, the only condition that is excluded here is, without loss of generality, when an inner vertex $v$ has in-degree at least two and out-degree at least three.
In this case, let $\mathrm{in}(v)=\{e_1\prec e_2\prec \cdots \prec e_\ell\}$ and let $\mathrm{out}(v)=\{f_1\prec f_2\prec \cdots \prec f_k\}$.
By assumption, $f_2$ is not idle, and thus $f_2$ must lie on an exceptional route $R$ by Theorem~\ref{thm:DKKample}.
Suppose edge $e_i$ also lies on $R$; note that one or both of $e_{i-1}$ or $e_{i+1}$ exist.
If $e_{i+1}$ exists, then create a route $R'$ that passes through $e_{i+1}$ and $f_1$ by extending $R'$ back from $e_{i+1}$ to a source and forward from $f_1$ to a sink.
In this case, $R$ is in conflict with $R'$ at $v$.
If $e_{i-1}$ exists, then a similar argument using $f_3$ will produce a route $R'$ that is in conflict with $R$ at $v$.
Note that $f_3$ exists since $k\geq 3$ by assumption.
In either case, we arrive at a contradiction to the fact that $R$ is exceptional, and hence to the assumption regarding the in-degree and out-degree of $v$.
\end{proof}

Lemma~\ref{lem:amplemeansfull} motivates the following definition.

\begin{definition}\label{def:fullgraph}
Let $G$ be a DAG.
For an inner vertex $v$, we say $v$ is \emph{full} if $\indeg(v)=2=\outdeg(v)$.
If every inner vertex of $G$ is full, then we say $G$ is \emph{full}.
\end{definition}

Thus, Lemma~\ref{lem:amplemeansfull} shows that every DAG with no idle edges that admits an ample framing must be full. An example of a full DAG is given in Figure~\ref{fig:nu-graph}.

\begin{lemma}\label{lem:abovebelowexceptionals}
Let $R$ be an exceptional route in a full DAG $G$ with ample framing $F$, and let $I(R)$ be the set of inner vertices of $G$ on $R$.
Then either $R$ passes through every vertex in $I(R)$ on the largest edges in the linear orders for $F$ or else $R$ passes through every vertex in $I(R)$ on the smallest edges in the linear orders for $F$.
\end{lemma}

\begin{proof}
Let $R$ be an exceptional route in $G$.
Let $v$ be the first inner vertex reached by $R$ after leaving a source, and let $\mathrm{in}(v)=\{e_1\prec e_2\}$ and $\mathrm{out}(v)=\{f_1\prec f_2\}$.
It is immediate that $R$ cannot contain $e_1$ and $f_2$, as in this case a conflicting route to $R$ can be constructed using $e_2$ and $f_1$, and vice versa.
Without loss of generality, suppose that $R$ contains $e_1$ and $f_1$.
Suppose that $w$ is an inner vertex on $R$ such that $\mathrm{in}(w)=\{e'_1\prec e'_2\}$ and $\mathrm{out}(w)=\{f'_1\prec f'_2\}$ where $R$ contains $e'_2$ and $f'_2$; suppose further that $w$ is the first such vertex on $R$ reached after $R$ leaves $v$.
Construct a new route $R'$ that starts with $e_2$ then follows $R$ to $w$ at which point it leaves on $f'_1$ and proceeds until terminating at a sink.
Then $R$ and $R'$ are in conflict at both $v$ and $w$, contradicting the assumption that $R$ is exceptional.
Hence, no such $w$ exists.
\end{proof}

\begin{theorem}\label{thm:uniqueexceptional}
If $G$ is a full DAG with ample framing $F$, then every edge in $G$ lies on a unique exceptional route.
\end{theorem}

\begin{proof}
Let $e$ be an edge in $G$.
By Theorem~\ref{thm:DKKample}, since $F$ is an ample framing, $e$ lies on some exceptional route $R$.
By Lemma~\ref{lem:abovebelowexceptionals}, either $e$ is first in the linear order for both of its vertices, or it is second in the linear order for both of its vertices.
Without loss of generality, assume $e$ is first in the order.
Any exceptional route that contains $e$ must pass through every inner vertex using the first edges in the linear orders associated to that vertex.
But, this constraint uniquely determines $R$.
\end{proof}

\begin{lemma}\label{lem:indegoutdeg}
If $G$ is a DAG satisfying $\indeg(v)=\outdeg(v)$ for every inner vertex $v$, then 
\[
\sum_{v\text{ a source of }G}\outdeg(v)=\sum_{w\text{ a sink of }G}\indeg(w) \, .
\]
\end{lemma}

\begin{proof}
The result follows from the following observation, which arises by canceling $1$'s and $-1$'s at inner vertices:
\[
0 = \sum_{(v,w)\,\in\, E}(-1+1) = \sum_{(v,w)\,:\, v \text{ a source}}(-1)+\sum_{(v,w)\,:\, w \text{ a sink}} 1 \, .
\]
\end{proof}

\begin{corollary}\label{cor:outdegreeexceptionalcount}
If $G$ is a full DAG with ample framing $F$, then the number of exceptional routes in $G$ is equal to
\begin{equation}\label{eq:sourcedeg}
\sum_{v\text{ a source of }G}\outdeg(v)=\sum_{w\text{ a sink of }G}\indeg(w) \, .
\end{equation}
Thus, the number of exceptional routes in an amply framed full DAG is independent of the framing.
\end{corollary}

\begin{proof}
The equality in~\eqref{eq:sourcedeg} follows from Lemma~\ref{lem:indegoutdeg}.
By Theorem~\ref{thm:uniqueexceptional}, each edge adjacent to a source in $G$ is contained in a unique exceptional route in $G$, and this correspondence is bijective.
Thus, the number of exceptional routes in $G$ is given by~\eqref{eq:sourcedeg}.
\end{proof}

Lemma~\ref{lem:amplemeansfull} shows that DAGs admitting ample framings must be full, Theorem~\ref{thm:uniqueexceptional} demonstrates that the exceptional routes are highly constrained in amply framed full DAGs, and Corollary~\ref{cor:outdegreeexceptionalcount} establishes that the number of exceptional routes for an ample framing is a function of the full DAG $G$ rather than the framing itself.
This naturally leads to the question of which collections of routes can form an exceptional set of routes in an ample framing of a full DAG.
The following definition and theorem answer this question.

\begin{definition}\label{def:adjacency}
Let $X$ be a set of routes in a DAG $G$ with framing $F$.
Define the \emph{adjacency graph} of $X$, $\Adj(G,X)$, to be the graph with vertex set $X$ where two routes $R,S\in X$ form an edge in $\Adj(G,X)$ if there exists a full vertex $v$ that lies on both $R$ and $S$.
\end{definition}

\begin{example}\label{ex:adjacencygraph}
Consider the graph $H$ from Figure~\ref{fig:nu-graph}, using the length framing given in Example~\ref{ex:car8framing}.
If $X=\{123456,136,146,1236\}$, then $\Adj(H,X)$ is given by the graph shown in Figure~\ref{fig:adjacencyexample}.
\end{example}

\begin{figure}
    \centering
\begin{tikzpicture}
\begin{scope}[scale=1, xshift=0, yshift=0]
	\vertex[fill,label=below:\footnotesize{$123456$}](a1) at (0,0) {};
	\vertex[fill,label=below:\footnotesize{$136$}](a2) at (2,0) {};
	\vertex[fill,label=above:\footnotesize{$1236$}](a3) at (0,2) {};
	\vertex[fill,label=above:\footnotesize{$146$}](a4) at (-2,2) {};
	\draw[thick] (a1)--(a2);
	\draw[thick] (a1)--(a3);
	\draw[thick] (a1)--(a4);
	\draw[thick] (a2)--(a3);
\end{scope}
\end{tikzpicture}
    \caption{The adjacency graph from Example~\ref{ex:adjacencygraph} with vertices consisting of the four non-exceptional routes listed in the right-hand column of Figure~\ref{fig:exceptional routes of contracted car(8)}.}
    \label{fig:adjacencyexample}
\end{figure}
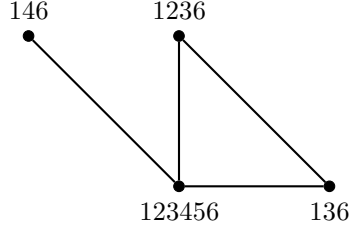

\begin{theorem}\label{thm:ampleclassification}
Let $G$ be a full DAG and let $X$ be a set of routes in $G$.
Then there exists an ample framing $F$ with exceptional set $X$ if and only if every edge in $G$ is contained in a unique route in $X$ and $\Adj(G,X)$ is bipartite.
\end{theorem}

\begin{proof}
Assume that $F$ is ample and let $\calE$ denote the set of exceptional routes.
Since $G$ is full, by Theorem~\ref{thm:uniqueexceptional} every edge belongs to a unique exceptional route.
For every exceptional route $R$ in $G$ with respect to $F$, Lemma~\ref{lem:abovebelowexceptionals} implies that $R$ can be labeled as either ``first'' or ``second'', depending on how $R$ passes through the linear orders of inner vertices.
Two routes with the same label cannot intersect at a full vertex $v$, as the two unique exceptional routes passing through $v$ must use both of the labels ``first'' and ``second".
Thus, in $\Adj(G,\calE)$, the labels ``first'' and ``second'' induce a bipartition of the routes.

For the converse, assume that $X$ is a set of routes such that every edge in $G$ is contained in a unique route and $\Adj(G,X)$ is bipartite with bipartition $A\uplus B$.
Label the routes in $A$ as ``first'' and the routes in $B$ as ``second''.
We construct a framing $F$ as follows.
Let $v$ be an inner vertex in $G$, with incoming edges $e_1$ and $e_2$ and outgoing edges $h_1$ and $h_2$.
Without loss of generality, suppose that $e_i$ and $h_i$ lie on a common route $R_i$ in $X$ for $i=1,2$.
Since $v$ is a full inner vertex common to both $R_1$ and $R_2$, it follows that $R_1$ and $R_2$ are not both in $A$ and are not both in $B$; hence each route has a distinct label.
Define the linear orders for $F$ on $\{e_1,e_2\}$ and $\{h_1,h_2\}$ by placing them first and second according to the labels on $R_1$ and $R_2$.
Since every edge in $G$ is in a unique route in $X$, linear orders for $F$ exist at every inner vertex and are well-defined.
Having constructed $F$, we finish by showing that $X$ is the set of exceptional routes for $F$, from which ampleness of $F$ follows from Theorem~\ref{thm:DKKample}.
Note that any exceptional route in $F$ is uniquely defined by passing through a specific edge with a given label, and every route in $X$ arises by passing through a specific edge with a given label.
Thus, $X$ is the set of exceptional routes for $F$, completing the proof.
\end{proof}

\begin{example}\label{ex:adjacencyclassification}
Continuing Example~\ref{ex:adjacencygraph}, note that the set $X$ cannot be an exceptional set because both the adjacency graph is not bipartite and the edge $15$ is not contained in any route in $X$. However, the set $\{123456,126,136,146,156\}$ contains every edge and has a bipartite adjacency graph; hence, it is the exceptional set for an ample framing.
\end{example}

\begin{corollary}\label{cor:amplevialabeling}
Suppose that $G$ is a full DAG and the edges of $G$ are labeled by the set $\{1,2\}$ such that if $e_1$ and $e_2$ are edges that are adjacent at an inner vertex, then the labels of $e_1$ and $e_2$ are distinct.
Then $G$ has an ample framing $F$ where the exceptional routes in $G$ consist of edges with constant labels. 
Conversely, any ample framing of $G$ induces such a labeling $~{\omega:E \rightarrow \{1,2\}}$ on the edges of $G$, given explicitly by
\begin{equation}\label{eq:weight}
\omega(e)= \begin{cases}
1, & \hbox{if $e=(v_1,v_2)$ is minimal in $\Out(v_1)$ and $\In(v_2)$},\\
2, & \hbox{if $e=(v_1,v_2)$ is maximal in $\Out(v_1)$ and $\In(v_2)$}.
\end{cases}
\end{equation}
\end{corollary}

\begin{proof}
Every edge $e$ leaving a source in $G$ determines a route in $G$ obtained by following the edges with the same label as $e$, yielding a set of routes $X$.
Every edge is contained in a unique route of this type, and the labeling implies that $\Adj(G,X)$ is bipartite.
The result then follows from Theorem~\ref{thm:ampleclassification}.
Given an ample framing of $G$, the two sets in the bipartition of $\Adj(G,X)$ can be labeled $1$ and $2$, inducing a labeling of the edges as described in the corollary.
\end{proof}

We know that if $G$ is a DAG without idle edges that admits an ample framing, then $G$ must be full.
Our next goal is to prove that every full DAG admits at least one ample framing.

\begin{theorem}\label{thm:ampleexistence}
If $G$ is a full DAG, then $G$ admits an ample framing.
\end{theorem}

\begin{proof}
We claim that $G$ can be written as an edge-disjoint union of DAGs of the following types: even cycles where the direction of the edges in the cycle are alternating, paths that start at a source or sink and end at a source or sink where the direction of the edges in the path are alternating, and edges from a source to a sink.
Given such a decomposition of $G$, any labeling of the paths and cycles that alternates $1$'s and $2$'s induces a labeling satisfying Corollary~\ref{cor:amplevialabeling}, and the result follows.

Let $\calO$ be a linear extension for $G$, i.e., a linear ordering $\calO=\{v_1<v_2<\cdots <v_n\}$ on the vertices of $G$ such that for any $v_i<v_j$ in $\calO$ we have that $(v_j,v_i)$ is not a directed edge in $G$.
We also assume that the sources of $G$ form the initial segment of $\calO$ and the sinks in $G$ form the terminal segment in $\calO$.
Every finite DAG admits a linear extension, as can be shown via induction by iteratively removing sinks from $G$ and adding it to the linear order.
Each $v_i$ in $\calO=\{v_1< \cdots <v_n\}$ that is an inner vertex in $G$ is associated to a unique length-two path in $G$, $P_i:=v_{j_1},v_i,v_{j_2}$, where $(v_{j_1},v_i)$ and $(v_{j_2},v_i)$ are both edges in $G$.
Note that the edges in $P_i$ alternate in direction.
Let $G_{i-1}$ denote the union of $P_j$ for all $j<i$, and assume by induction that $G_{i-1}$ has a decomposition of our desired type into paths and cycles that consist of sequences of $P_j$'s connected at their degree-one inner-vertex endpoints.
It is immediate that $G_1=P_1$ satisfies this, and thus we have a base case.
For the $i$-th step, consider the edges $(v_{j_1},v_i)$ and $(v_{j_2},v_i)$ that make up $P_i$.
One of the following must hold:
\begin{enumerate}
    \item Each endpoint of $P_i$ is an inner vertex, and these endpoints are the two endpoints of a single path in the decomposition of $G_{i-1}$ with ending edges from $\mathrm{out}(v_{j_1})$ and $\mathrm{out}(v_{j_2})$.
    Attaching $P_i$ to this path forms a cycle in $G_i$.
    \item One of the two endpoints of $P_i$ is adjacent to a degree-one inner vertex in a path in the decomposition of $G_{i-1}$ that includes an edge in either $\mathrm{out}(v_{j_1})$ or $\mathrm{out}(v_{j_2})$.
    In this case, $P_i$ extends the path by two more edges.
    \item Both of the two endpoints of $P_i$ are adjacent to degree-one inner vertices in two distinct paths in the decomposition of $G_{i-1}$, each of which include one edge from $\mathrm{out}(v_{j_1})$ or $\mathrm{out}(v_{j_2})$.
    In this case, the concatenation of the two existing paths via $P_i$ forms a single path.
    \item Neither of the other edges in $\mathrm{out}(v_{j_1})$ and $\mathrm{out}(v_{j_2})$ are contained in $G_{i-1}$.
    In this case, $P_i$ is added to the decomposition of $G_{i-1}$, yielding a decomposition of $G_i$. 
\end{enumerate}
In any of these cases, the result is that the length of the paths and cycles in the disjoint decomposition is even, since we are always appending a path of length two at each step.
Once $G_n$ is obtained, then $G$ is formed by adding the sinks of $G$ and then extending any path in $G_n$ that terminates at an inner vertex $v$ by the edge from $v$ to a sink of $G$, and also by creating length-two paths for any pairs of multiedges that connect a sink in $G$ to the same inner vertex.
\end{proof}

\begin{figure}
\begin{center}
\begin{tikzpicture}
\begin{scope}[scale=1, xshift=0, yshift=0]
    \vertex[fill,label=below:\footnotesize{$s$}](as) at (0,0) {};
	\vertex[fill,label=below:\footnotesize{$1$}](a1) at (1,0) {};
	\vertex[fill,label=below:\footnotesize{$2$}](a2) at (2,0) {};
	\vertex[fill,label=below:\footnotesize{$3$}](a3) at (3,0) {};
	\vertex[fill,label=below:\footnotesize{$4$}](a4) at (4,0) {};
	\vertex[fill,label=below:\footnotesize{$5$}](a5) at (5,0) {};
	\vertex[fill,label=below:\footnotesize{$6$}](a6) at (6,0) {};
	\vertex[fill,label=below:\footnotesize{$7$}](a7) at (7,0) {};
	\vertex[fill,label=below:\footnotesize{$8$}](a8) at (8,0) {};
	\vertex[fill,label=below:\footnotesize{$9$}](a9) at (9,0) {};
	\vertex[fill,label=below:\footnotesize{$X$}](a10) at (10,0) {};
	\vertex[fill,label=below:\footnotesize{$t$}](at) at (11,0) {};

	\draw[-stealth, thick] (as) to[out=60,in=120] (a1);
	\draw[-stealth, thick] (as) to[out=-60,in=-120] (a1);
	\draw[-stealth, thick] (as) to[out=60,in=120] (a2);
	\draw[-stealth, thick] (as) to[out=-60,in=-120] (a2);
	\draw[-stealth, thick] (as) to[out=60,in=120] (a3);
	\draw[-stealth, thick] (as) to[out=-60,in=-120] (a3);
	\draw[-stealth, thick] (as) to[out=-60,in=-120] (a7);
	\draw[-stealth, thick] (as) to[out=60,in=120] (a6);
	\draw[-stealth, thick] (as) to[out=-60,in=-120] (a10);
	\draw[-stealth, thick] (a1) to[out=-60,in=-120] (a4);
	\draw[-stealth, thick] (a1) to[out=60,in=120] (a7);
	\draw[-stealth, thick] (a2) to[out=60,in=120] (a4);
	\draw[-stealth, thick] (a2) to[out=-60,in=-120] (a5);
	\draw[-stealth, thick] (a3) to[out=60,in=120] (a5);
	\draw[-stealth, thick] (a3) to[out=-60,in=-120] (a6);
	\draw[-stealth, thick] (a4) to[out=60,in=120] (at);
	\draw[-stealth, thick] (a4) to[out=-60,in=-120] (a9);
	\draw[-stealth, thick] (a5) to[out=60,in=120] (a8);
	\draw[-stealth, thick] (a5) to[out=-60,in=-120] (at);
	\draw[-stealth, thick] (a6) to[out=60,in=120] (a10);
	\draw[-stealth, thick] (a6) to[out=-60,in=-120] (at);
	\draw[-stealth, thick] (a7) to[out=-60,in=-120] (a8);
	\draw[-stealth, thick] (a7) to[out=60,in=120] (a9);
	\draw[-stealth, thick] (a8) to[out=60,in=120] (at);
	\draw[-stealth, thick] (a8) to[out=-60,in=-120] (at);
	\draw[-stealth, thick] (a9) to[out=60,in=120] (at);
	\draw[-stealth, thick] (a9) to[out=-60,in=-120] (at);
	\draw[-stealth, thick] (a10) to[out=60,in=120] (at);
	\draw[-stealth, thick] (a10) to[out=-60,in=-120] (at);
\end{scope}
\end{tikzpicture}
\end{center}
\caption{The graph $G$ from Example~\ref{ex:disjointpathscycles}.}
\label{fig:disjointpathscycles}
\end{figure}
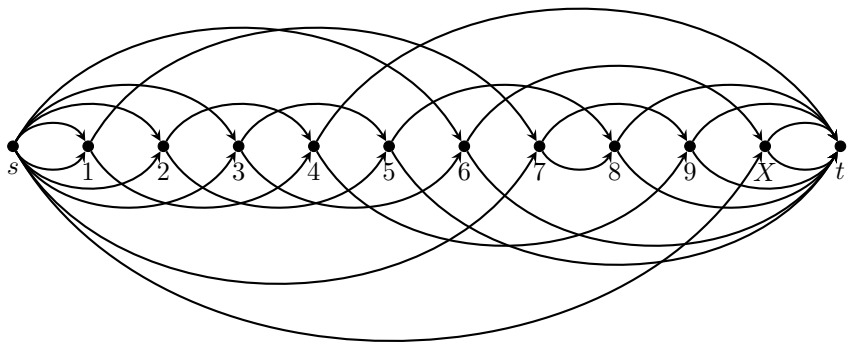

\begin{example}\label{ex:disjointpathscycles}
Let $G$ be the DAG in Figure~\ref{fig:disjointpathscycles}, with the inner vertices in the linear extension used in the proof of Theorem~\ref{thm:ampleexistence} read from left to right in the figure.
The disjoint cycle and path decomposition of $G$ arising from the algorithm in the proof of Theorem~\ref{thm:ampleexistence} is:
\[
s1s, s2s, s3s, s7142536s, t58794t, sX6t, t8t, t9t, tXt
\, .\]
\end{example}

Combining Lemma~\ref{lem:amplemeansfull} and Theorem~\ref{thm:ampleclassification}, we obtain the following corollary.

\begin{corollary}\label{cor:amplecharacterizationiff}
A DAG $G$ with no idle edges admits an ample framing if and only if $G$ is full.
\end{corollary}


\subsection{Valid Framed Graphs}

Having handled the case where $G$ has no idle edges, we now consider the case where $G$ contains idle edges.

\begin{definition}\label{def:fullcontraction}
Given a DAG $G$, suppose that $G$ admits a complete contraction $H$ such that $H$ is full.
In this case we call $H$ a \emph{full contraction} of $G$.
If $G$ admits a full contraction, then we say $G$ is \emph{valid}.
\end{definition}

Observe that the graph in Figure~\ref{fig:nu-graph} is a full contraction of $\car(8)$, and thus $\car(8)$ is valid.
Starting with a full graph, we can construct valid graphs by reversing the effects of idle-edge contractions, leading to the following definition.

\begin{definition}\label{def:idleedgeexpansion}
If $G$ and $G'$ are DAGs such that $G$ is obtained from $G'$ by contracting an idle edge, then we say that $G'$ is an \emph{idle-edge expansion} of $G$.
If $G=G^1,G^2,G^3,\ldots,G^m$ is a sequence of graphs such that $G^i$ is an idle-edge expansion of $G^{i-1}$ for every $i$, where $e_i$ is the new edge introduced to $G^{i-1}$, we call $G^m$ an \emph{idle expansion} of $G$.
If $e_j$ is one of these idle edges such that there exists a directed path $e_{i_1}e_{i_2}\cdots e_{i_k}e_j$ in $G^m$ where $e_{i_1}$ leaves a source, then we say $e_j$ is \emph{source-reachable} with respect to the idle edge expansion producing $G^m$; we define \emph{sink-reachable} idle edges similarly.
If no such directed path to $e_j$ from a source or sink exists in $G^m$, then we say that $e_j$ is an \emph{inner idle edge} with respect to the idle expansion producing $G^m$.
\end{definition}

By definition, every valid graph is obtained as an idle expansion of a full graph.
This leads to the idle edges in a valid graph having restricted structure.

\begin{proposition}\label{prop:idleforest}
Given a valid graph $G$, the set of idle edges forms a forest in $G$.
Further, the components of the forest containing source-reachable idle edges are rooted at the sources of $G$, and similarly for sink-reachable idle edges.
\end{proposition}

\begin{proof}
We go by induction on the number of idle edges.
Starting from a full graph $H$, a single idle-edge expansion creates a valid graph with a single idle edge, which is a forest.
Now assume that $G$ is a valid graph whose idle edges form a forest in $G$.
An idle-edge expansion of $G$ is obtained by selecting a single vertex $v$ in $G$ and expanding that vertex to become an edge $e=(w_1,w_2)$, where the set of outgoing edges from $v$ are split into outgoing edges from $w_1$ and $w_2$, and similarly for the incoming edges to $v$.
If $v$ is not an endpoint of an idle edge in $G$, then the new edge $e$ is a new component of the forest in the resulting idle-edge expansion of $G$ by $e$.
If $v$ is a vertex in the forest of idle edges in $G$, then expanding $v$ to $e$ results in a larger forest of idle edges in the expansion of $G$ by $e$.
\end{proof}

\begin{proposition}\label{prop:validframingexpansion}
Suppose $[G,F]$ is a valid DAG and $F$ is ample.
If $H$ is an idle-edge expansion of $G$ by $e$, then $H$ admits an ample framing with the property that there is a bijection between the exceptional routes in $G$ and the exceptional routes in $H$.
Further, every ample framing of $H$ collapses to an ample framing of $G$ with this property.
\end{proposition}

\begin{proof}
Since the source-reachable idle edges in $G$ form a forest with components rooted at sources, there is a unique path from any source to a leaf of such a component.
Thus, there is only one edge entering a vertex in such a component, and the linear order on the outgoing edges of any such vertex is irrelevant to whether or not a route is exceptional.
The situation is similar for sink-reachable idle edges.
If $e=(w_1,w_2)$ is expanded from the vertex $v$ in $G$, we know that the set of outgoing edges from $v$ are split into outgoing edges from $w_1$ and $w_2$, and similarly for the incoming edges to $v$.
If that vertex is in a source- or sink-reachable idle edge, then any linear order on the outgoing/incoming edges of $w_1$ and $w_2$ will extend $F$ to a framing of $H$ where the exceptional routes in $H$ are the same as in $G$ except that any routes passing through $v$ now include $e$.
If the vertex is in an inner idle edge, then we use the same linear orders for the incoming and outgoing edges of $w_1$ and $w_2$ that are used for $\mathrm{in}(v)$ and $\mathrm{out}(v)$, and similarly any exceptional route in $G$ passing through $v$ extends uniquely to an exceptional route in $H$ containing $e$.
When $e$ is contracted from $H$ to $G$, we can merge the framing orders to preserve exceptional routes.
\end{proof}

By iteratively applying Proposition~\ref{prop:validframingexpansion} starting from a full DAG, we obtain the following corollary.

\begin{corollary}
Every valid DAG~$G$ admits an ample framing.
Further, if $H$ is a full contraction of $G$, then there is a bijection between the set of DKK triangulations of $\calF_+(G)$ and the set of DKK triangulations of $\calF_+(H)$.
\end{corollary}


\section{Enumerating Ample Framings}\label{sec:enumerating}

In this section, we consider the problem of counting the ample framings for valid DAGs.
We also determine the number of ample framings for several special classes of DAGs.

\begin{theorem}\label{thm:poweroftwo}
In the decomposition of a full DAG $G$ into edge-disjoint even cycles and paths satisfying the conditions in the proof of Theorem~\ref{thm:ampleexistence}, suppose that there are $M$ paths and cycles that contain at least one inner vertex, i.e., $M$ paths and cycles excluding edges from a source to a sink.
Then there are $2^M$ ample framings of $G$.
\end{theorem}

\begin{proof}
Every ample framing induces a $\{1,2\}$-labeling of the paths and cycles in the decomposition from Theorem~\ref{thm:ampleexistence} that is alternating.
Since there are exactly $2^M$ such labelings, and each of those labelings induces an ample framing, the result follows.
\end{proof}

\begin{example}\label{ex:enumeratefullample}
In the full DAG from Example~\ref{ex:disjointpathscycles}, there are nine disjoint paths and cycles in $G$.
Thus, there are $2^9$ ample framings.
\end{example}

The following corollary enumerates the distinct DKK triangulations of $\calF_1(G)$.

\begin{corollary}
\label{cor:countDKKfull}
Suppose that $G$ is a full DAG and there are $M$ paths and cycles containing at least one inner vertex in the decomposition of $G$ into edge-disjoint even cycles and paths satisfying the conditions in the proof of Theorem~\ref{thm:ampleexistence}.
Then there are $2^{M-1}$ distinct DKK triangulations of $\calF_1(G)$.
\end{corollary}

\begin{proof}
Given an ample framing $F$ of $G$, exchanging all $1$'s for $2$'s and vice versa yields another ample framing with the same triangulation.
For any other pair of ample framings $F_1$ and $F_2$, the sets of exceptional routes corresponding to $F_1$ and $F_2$ are distinct and hence the DKK triangulations for $F_1$ and $F_2$ are distinct.
\end{proof}

\begin{corollary}\label{cor:poweroftwovalid}
Suppose that $G$ is a valid DAG and $H$ is a full contraction of $G$.
Suppose that there are $M$ paths and cycles containing at least one inner vertex in the decomposition of $H$ into edge-disjoint even cycles and paths satisfying the conditions in the proof of Theorem~\ref{thm:ampleexistence}.
Write $V_1$ for the set of non-source vertices that are endpoints of source-reachable idle edges in $G$, and write $V_2$ for the set of non-sink vertices that are endpoints of sink-reachable idle edges in $G$.
The number of ample framings of $G$ is equal to
\[
2^M\prod_{v\in V_1}|\mathrm{out}(v)|!\prod_{v\in V_2}|\mathrm{in}(v)|! \, .
\]
\end{corollary}

\begin{proof}
By Proposition~\ref{prop:validframingexpansion}, any ample framing of $G$ descends to an ample framing of $H$.
Given an ample framing $F$ of $H$, any linear order of the outgoing edges from $V_1$ in $G$ and any linear order of the incoming edges to $V_2$ in $G$ will extend $F$ to an ample framing of $G$. 
Thus, we have that each ample framing of $H$ extends to 
\[
\prod_{v\in V_1}|\mathrm{out}(v)|!\prod_{v\in V_2}|\mathrm{in}(v)|!
\]
ample framings of $G$.
\end{proof}

As an application of Corollary~\ref{cor:poweroftwovalid}, we consider the following class of DAGs.
Note that these DAGs have been previously studied with regard to their flow polytopes; see Remark~\ref{rem:Gkn+k} for details.

\begin{definition}\label{def:gkn}
Let $1\leq k\leq n$ be integers.
Define $G(k,n+1)$ to be the DAG with vertex set $[n+1]$ and directed edges $\{(i,i+1):i\in [n]\}\cup\{(i,i+k):i\in [n-k+1]\}$.
\end{definition}

\begin{example}\label{ex:ampleframingsvalidexample}
Figure~\ref{fig:g310} depicts $G(3,10)$.
A full contraction $H$ is shown in Figure~\ref{fig:g310fullcontraction}.
$H$ admits the following decomposition into four paths and cycles:
\[
141, 15476(10), 165(10), (10)7(10) \, .
\]
Further, in $G(3,10)$ the vertices $2$ and $3$ are source-reachable while $8$ and $9$ are sink-reachable, and the out- and in-degree of each of these, respectively, is equal to $2$.
Thus, by Corollary~\ref{cor:poweroftwovalid}, $G(3,10)$ has $2^42^22^2=2^8$ ample framings.
\end{example}

\begin{figure}
\begin{center}
\begin{tikzpicture}
\begin{scope}[scale=1, xshift=0, yshift=0]
	\vertex[fill,label=below:\footnotesize{$1$}](a1) at (1,0) {};
	\vertex[fill,label=below:\footnotesize{$2$}](a2) at (2,0) {};
	\vertex[fill,label=below:\footnotesize{$3$}](a3) at (3,0) {};
	\vertex[fill,label=below:\footnotesize{$4$}](a4) at (4,0) {};
	\vertex[fill,label=below:\footnotesize{$5$}](a5) at (5,0) {};
	\vertex[fill,label=below:\footnotesize{$6$}](a6) at (6,0) {};
	\vertex[fill,label=below:\footnotesize{$7$}](a7) at (7,0) {};
	\vertex[fill,label=below:\footnotesize{$8$}](a8) at (8,0) {};
	\vertex[fill,label=below:\footnotesize{$9$}](a9) at (9,0) {};
	\vertex[fill,label=below:\footnotesize{$10$}](a10) at (10,0) {};

	\draw[-stealth, thick] (a1) to (a2);
	\draw[-stealth, thick] (a2) to (a3);
	\draw[-stealth, thick] (a3) to (a4);
	\draw[-stealth, thick] (a4) to (a5);
	\draw[-stealth, thick] (a5) to (a6);
	\draw[-stealth, thick] (a6) to (a7);
	\draw[-stealth, thick] (a7) to (a8);
	\draw[-stealth, thick] (a8) to (a9);
	\draw[-stealth, thick] (a9) to (a10);
	\draw[-stealth, thick] (a1) to[out=60,in=120] (a4);
	\draw[-stealth, thick] (a2) to[out=60,in=120] (a5);
	\draw[-stealth, thick] (a3) to[out=60,in=120] (a6);
	\draw[-stealth, thick] (a4) to[out=60,in=120] (a7);
	\draw[-stealth, thick] (a5) to[out=60,in=120] (a8);
	\draw[-stealth, thick] (a6) to[out=60,in=120] (a9);
	\draw[-stealth, thick] (a7) to[out=60,in=120] (a10);
\end{scope}
\end{tikzpicture}
\end{center}
\caption{$G(3,10)$}
\label{fig:g310}
\end{figure}

\begin{figure}
\begin{center}
\begin{tikzpicture}
\begin{scope}[scale=1, xshift=0, yshift=0]
	\vertex[fill,label=below:\footnotesize{$1$}](a1) at (1,0) {};
	\vertex[fill,label=below:\footnotesize{$4$}](a4) at (4,0) {};
	\vertex[fill,label=below:\footnotesize{$5$}](a5) at (5,0) {};
	\vertex[fill,label=below:\footnotesize{$6$}](a6) at (6,0) {};
	\vertex[fill,label=below:\footnotesize{$7$}](a7) at (7,0) {};
	\vertex[fill,label=below:\footnotesize{$10$}](a10) at (10,0) {};

	\draw[-stealth, thick] (a1) to (a4);
	\draw[-stealth, thick] (a4) to (a5);
	\draw[-stealth, thick] (a5) to (a6);
	\draw[-stealth, thick] (a6) to (a7);
	\draw[-stealth, thick] (a7) to (a10);
	\draw[-stealth, thick] (a1) to[out=60,in=120] (a4);
	\draw[-stealth, thick] (a1) to[out=60,in=120] (a5);
	\draw[-stealth, thick] (a1) to[out=60,in=120] (a6);
	\draw[-stealth, thick] (a4) to[out=60,in=120] (a7);
	\draw[-stealth, thick] (a5) to[out=60,in=120] (a10);
	\draw[-stealth, thick] (a6) to[out=60,in=120] (a10);
	\draw[-stealth, thick] (a7) to[out=60,in=120] (a10);
\end{scope}
\end{tikzpicture}
\end{center}
\caption{The graph $H$ obtained as a full contraction of $G(3,10)$.}
\label{fig:g310fullcontraction}
\end{figure}
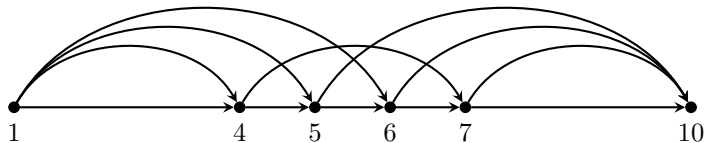

\begin{theorem}
The number of ample framings of $G(k,n+1)$ is:
\begin{itemize}
    \item $4^{n-k}$ for $n=k+1,\ldots,2k-1$,
    \item $2^n$ for $n=2k+1,\ldots,3k-1$, and
    \item $2^{3k-1}$ for $n\geq 3k$.
\end{itemize}
\end{theorem}

\begin{proof}
Suppose that $n\leq 2k-1$.
Then every edge of the form $(i,i+1)$ in $G(k,n+1)$ is an idle edge, except for the case where $n=2k-1$, in which case $G(k,2k)$ has a single non-idle edge from $k$ to $k+1$.
For any of these DAGs, the full contraction of $G$ has a single source, a single sink, and multiedges between these.
In this case, $M=0$ and Corollary~\ref{cor:poweroftwovalid} yields $2^{n-k}2^{n-k}$ ample framings of $G(k,n+1)$.

Suppose next that $n=2k,\ldots,3k-1$.
In this case, there are $2k-2$ idle edges having the form $(i,i+1)$ for $i=1,\ldots,k-1$ and $i=n-k+2,\ldots,n$. 
Suppose $H$ is the full contraction of $G(k,n+1)$ obtained by contracting these edges.
See for example Figure~\ref{fig:g411fullcontraction} illustrating $G(4,11)$ where $k=4$ and $n=10$.
In this case, the disjoint path and cycle decomposition of $H$ consists of the cycles
\begin{align*}
& 1(k+1)1\\
& (n+1)(k+1)(k+2)1\\
& (n+1)(k+2)(k+3)1\\ & (n+1)(k+3)(k+4)1\\
& \hspace{2.25cm} \vdots \\
& (n+1)(n-k)(n-k+1)1\\
& (n+1)(n-k+1)(n+1) \, ,
\end{align*}
and thus $M=n-2k+2$.
Hence, by Corollary~\ref{cor:poweroftwovalid}, the total number of ample framings is $2^{n-2k+2}2^{k-1}2^{k-1}=2^n$.

Finally, suppose that $n\geq 3k$.
In this case, there are again $2k-2$ idle edges, and we assume $H$ is the result of fully contracting these.
See for example Figure~\ref{fig:g310fullcontraction}, which shows a full contraction of $G(3,10)$.
In this case, it is straightforward to verify that the disjoint path and cycle decomposition contains $k+1$ paths and cycles.
Thus, the total number of ample framings is $2^{k+1}2^{k-1}2^{k-1}=2^{3k-1}$.
\end{proof}

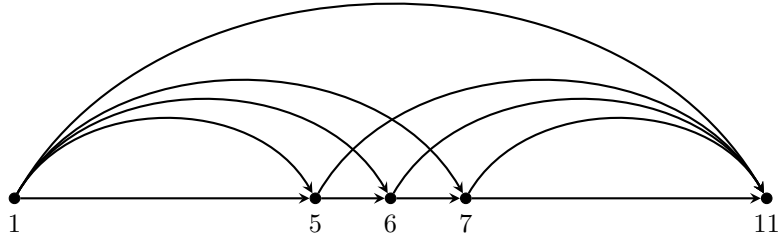
\begin{figure}
\begin{center}
\begin{tikzpicture}
\begin{scope}[scale=1, xshift=0, yshift=0]
	\vertex[fill,label=below:\footnotesize{$1$}](a1) at (1,0) {};
	\vertex[fill,label=below:\footnotesize{$5$}](a5) at (5,0) {};
	\vertex[fill,label=below:\footnotesize{$6$}](a6) at (6,0) {};
	\vertex[fill,label=below:\footnotesize{$7$}](a7) at (7,0) {};
	\vertex[fill,label=below:\footnotesize{$11$}](a11) at (11,0) {};

	\draw[-stealth, thick] (a1) to (a5);
	\draw[-stealth, thick] (a5) to (a6);
	\draw[-stealth, thick] (a6) to (a7);
	\draw[-stealth, thick] (a7) to (a11);
	\draw[-stealth, thick] (a1) to[out=60,in=120] (a5);
	\draw[-stealth, thick] (a1) to[out=60,in=120] (a6);
	\draw[-stealth, thick] (a1) to[out=60,in=120] (a7);
	\draw[-stealth, thick] (a1) to[out=60,in=120] (a11);
	\draw[-stealth, thick] (a5) to[out=60,in=120] (a11);
	\draw[-stealth, thick] (a6) to[out=60,in=120] (a11);
	\draw[-stealth, thick] (a7) to[out=60,in=120] (a11);
\end{scope}
\end{tikzpicture}
\end{center}
\caption{The graph $H$ obtained as a full contraction of $G(4,11)$.}
\label{fig:g411fullcontraction}
\end{figure}


\section{Bijection between $KQ/I$-modules and Non-exceptional Routes}\label{sec:bijection}

\subsection{Background on Path Algebras and Gentle Algebras}

A \emph{quiver} $Q=(Q_0, Q_1)$ is a finite directed graph where $Q_0$ denotes the set
of vertices and $Q_1$ denotes the set of arrows in $Q$.
Given an arrow $\alpha\in Q_1$, its starting and ending vertices are denoted by $s(\alpha)$, $t(\alpha)$ respectively, where $s(\alpha)\xrightarrow{\alpha} t(\alpha)$.
A \emph{path} of length $n$ in $Q$ is a composition of arrows $\alpha_1\alpha_2\cdots\alpha_n$ such that $t(\alpha_i)=s(\alpha_{i+1})$ for all $i = 2, \dots, n-1$.
In addition, for each vertex $i\in Q_0$, we define a \emph{constant path}, denoted $\varepsilon_i$, with $s(\varepsilon_i)=t(\varepsilon_i)=i$.
Constant paths are said to be of length zero. 

Let $K$ be an algebraically closed field.
The \emph{path algebra} over a quiver $Q$, denoted by $KQ$, is the $K$-algebra with basis given by the set of all paths in $Q$.  Moreover, multiplication is defined as concatenation of paths in $Q$.  

\begin{example}
Let $Q = 1\xrightarrow{\alpha}2 \xrightarrow{\beta} 3$ be a quiver.  Then the path algebra $KQ$ is a $K$-vector space with bases $\{\varepsilon_1, \varepsilon_2, \varepsilon_3, \alpha, \beta, \alpha\beta\}$.  Every constant path $\varepsilon_i$ is an idempotent of $KQ$, hence $\varepsilon_i^2=\varepsilon_i$.  Moreover, the product $\alpha\cdot \beta = \alpha\beta$ while $\beta\cdot \alpha=0$ since this does not correspond to a path in $Q$.  Similarly, we have $\varepsilon_1\cdot \alpha=\alpha$ while $\varepsilon_2\cdot\varepsilon_1=0$.  
\end{example}

Given an ideal $I$ of $KQ$, consisting of paths of length at least 2, we can also consider a quotient of the path algebra $KQ/I$.  Next, we define a special class of such algebras called gentle, which were originally introduced and studied in \cite{AssSko}.

\begin{definition}
\label{def:gentle}
Let $I$ be a monomial ideal, then we say that a finite dimensional algebra $\Lambda = KQ/I$ is a \emph{gentle algebra} if it satisfies the following properties:
\begin{itemize}
\item[(a)] for any vertex $i\in Q_0$, there are at most two incoming and at most two outgoing arrows,
\item[(b)] for any arrow $\alpha \in Q_1$, there is at most one arrow $\beta$ and at most one arrow $\gamma$ such that $\alpha\beta\not\in I$ and $\gamma\alpha\not\in I$,
\item[(c)] for each arrow $\alpha \in Q_1$, there is at most one arrow $\beta$ and at most one arrow $\gamma$ such that $0\not=\alpha\beta\in I$ and $0\not=\gamma\alpha\in I$,
\item[(d)] there exists a generating set for the ideal $I$ consisting of a finite set of paths of length two.
\end{itemize}
\end{definition}

Gentle algebras are especially nice, because their module categories are well-understood in terms of walks in the quiver \cite{ButlerRingel}, which we describe below.  Note that we are only considering finite dimensional modules here. 

Let $\Lambda = KQ/I$ be a gentle algebra. We formally define $Q_1^{-1}$ to be the set of inverse arrows of $Q$. Elements of $Q_1^{-1}$ are denoted by $\alpha^{-1}$, for $\alpha\in Q_1$, and $s(\alpha^{-1})   \coloneqq t(\alpha)$ and $t(\alpha^{-1}) \coloneqq s(\alpha)$. A \emph{string}, or equivalently a \emph{walk}, in $\Lambda$ of length $n$ is a word $w = \alpha_1^{t_1} \cdots \alpha_n^{t_n}$ in the alphabet $Q_1 \cup Q_1^{-1}$ with $t_i \in \{\pm1\}$, for all $i \in \{1,2,\dots ,n\}$, which satisfies the following conditions:
\begin{itemize}
\item[(a)] $t(\alpha_i^{t_i})=s(\alpha_{i+1}^{t_i+1})$ and $\alpha_{i+1}^{t_i+1}\not=\alpha_i^{-t_i}$, for all $i\in \{1,\dots ,n-1\}$,
\item[(b)] $w$ and also $w^{-1} \coloneqq \alpha_n^{-t_n}\cdots \alpha_1^{-t_1}$ do not contain a subpath in $I$. 
\end{itemize}
We refer to the symbols $\alpha_i$ and $\alpha^{-1}_j$ appearing in some string $w$ as \emph{arrows} and \emph{inverse arrows} of $w$.  In the case that $\alpha_i^{t_i}$ has $t_i = 1$, we will simply write $\alpha_i$. The constant path $\varepsilon_i$ of length zero is also considered to be a string.  Moreover, we consider strings up to the equivalence relation where a string $w$ is identified with $w^{-1}$.

We say that $v$ is a \emph{substring} of $w$ if $v=\alpha_i^{t_i}\cdots \alpha_j^{t_j}$ for some $1\leq i\leq j\leq n$ or if $v=\varepsilon_i$ for some  vertex $i \in Q_0$ through which $w$ passes. 
We say $w$ \emph{starts} at $s(w) = s(\alpha_1^{t_1})$ and \emph{ends} at $t(w) = t(\alpha_n^{t_n})$. 
Moreover, a string is called \emph{directed} if $t_i=1$ for all $i\in \{1, \dots, n\}$ or $t_i=-1$ for all $i\in \{1, \dots, n\}$.

For a gentle algebra $\Lambda$, there is an indecomposable $\Lambda$-module $M(w)$ associated to every string $w$.  Moreover, $\Lambda$ is of \emph{finite representation type}, meaning that there are only finitely many indecomposable $\Lambda$-modules up to isomorphism whenever there are only finitely many strings for $\Lambda$.  In this case, there is a bijection between strings $w$ and indecomposable $\Lambda$-modules $M(w)$ up to isomorphism.

\subsection{The Bijection}

In this section we describe a bijection between non-exceptional routes in a full DAG $G$ and indecomposable modules over certain gentle algebras.  As a consequence, we show that the dual graph of the DKK triangulation has a poset structure coming from the $\tau$-tilting poset of the associated gentle algebra. 

Let $G$ be a full DAG with a fixed ample framing.   
By Corollary~\ref{cor:amplevialabeling}, the framing induces a labeling of the edges of $G$ by either 1 or 2 as stated in Equation~\eqref{eq:weight}.
We refer to this labeling as the \emph{weight function} for the framing.

\begin{definition}
Let $R$ be a route in $G$ so that $R=(e_1,e_2,\ldots, e_k)$ is an ordered set of edges.  Define the weight of $R$ by $\omega(R) = (\omega(e_1), \omega(e_2),\ldots, \omega(e_k)) \in \{1,2\}^k$.
\end{definition}
Note this means that $R$ is exceptional if and only if $\omega(R)$ is a vector of all ones or all twos.
Next, we define a path algebra over a quiver with relations coming from $G$.
In the context of our work, every full DAG has an associated quiver with an associated algebra.

\begin{definition}\label{def:quiver}
Let $G$ be a full DAG with a fixed ample framing.
Define a quiver $Q$ whose vertices are the inner vertices of $G$ and whose arrows come from directed edges $e=(v_1,v_2)$ in $G$ such that 
\[\begin{cases}
\xymatrix{v_1 \ar[r]^e & v_2} \hbox{ in } Q, & \hbox{if $\omega(e) =1$}, \\
\xymatrix{v_1 & \ar[l]_e v_2} \hbox{ in } Q, & \hbox{if $\omega(e)=2$}.
\end{cases}
\]
Furthermore, define a set of relations on the path algebra $KQ$ as follows.  Let $I$ be the ideal of $KQ$ generated by all paths $e_1e_2: \xymatrix{v_1 \ar[r]^{e_1} & v_2 \ar[r]^{e_2} & v_3}$ in $Q$ such that $\omega(e_1) \neq \omega(e_2)$.
This defines a \emph{path algebra with relations} $\Lambda(G):=KQ/I$.
\end{definition}

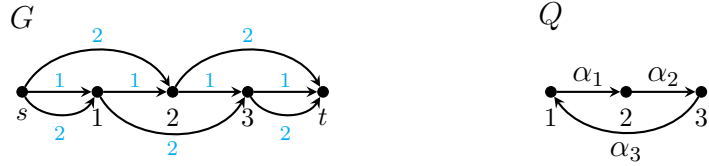
\begin{figure}
\begin{center}
    \begin{tikzpicture}
\begin{scope}[scale=1, xshift=0, yshift=0]
	\vertex[fill,label=below:\footnotesize{$s$}](as) at (1,0) {};
	\vertex[fill,label=below:\footnotesize{$1$}](a1) at (2,0) {};
	\vertex[fill,label=below:\footnotesize{$2$}](a2) at (3,0) {};
	\vertex[fill,label=below:\footnotesize{$3$}](a3) at (4,0) {};
	\vertex[fill,label=below:\footnotesize{$t$}](at) at (5,0) {};
	
	\draw[-stealth, thick] (as)--(a1);
	\draw[-stealth, thick] (a1)--(a2);
	\draw[-stealth, thick] (a2)--(a3);
	\draw[-stealth, thick] (a3)--(at);

	\draw[-stealth, thick] (as) to[out=60,in=120] (a2);
	\draw[-stealth, thick] (a2) to[out=60,in=120] (at);

	\draw[-stealth, thick] (as) to[out=-60,in=-120] (a1);
	\draw[-stealth, thick] (a1) to[out=-60,in=-120] (a3);
	\draw[-stealth, thick] (a3) to[out=-60,in=-120] (at);
	
    \node[] at (1.5,0.15) {\textcolor{cyan}{\tiny$1$}}; 
    \node[] at (2.5,0.15) {\textcolor{cyan}{\tiny$1$}}; 
    \node[] at (3.5,0.15) {\textcolor{cyan}{\tiny$1$}}; 
    \node[] at (4.5,0.15) {\textcolor{cyan}{\tiny$1$}}; 
    
    \node[] at (2,0.75) {\textcolor{cyan}{\tiny$2$}}; 
    \node[] at (4,0.75) {\textcolor{cyan}{\tiny$2$}}; 
    
    \node[] at (1.5,-0.55) {\textcolor{cyan}{\tiny$2$}}; 
    \node[] at (3, -0.75)  {\textcolor{cyan}{\tiny$2$}}; 
    \node[] at (4.5,-0.55) {\textcolor{cyan}{\tiny$2$}}; 
    
    \node[] at (1,1) {$G$};

\end{scope}

\begin{scope}[scale=1, xshift=200, yshift=0]
	\vertex[fill,label=below:\footnotesize{$1$}](a1) at (1,0) {};
	\vertex[fill,label=below:\footnotesize{$2$}](a2) at (2,0) {};
	\vertex[fill,label=below:\footnotesize{$3$}](a3) at (3,0) {};

	\draw[-stealth, thick] (a1)--(a2);
	\draw[-stealth, thick] (a2)--(a3);
	\draw[-stealth, thick] (a3) to[out=-120,in=-55] (a1);
	
    \node[] at (1.5,0.2) {$\alpha_1$};
    \node[] at (2.5,0.2) {$\alpha_2$};
    \node[] at (2,-0.8) {$\alpha_3$};
    
    \node[] at (1,1) {$Q$};

\end{scope}
\end{tikzpicture}
\end{center}
    \caption{The full DAG \(G\) is the complete contraction of \(G(2,7)\) with an ample framing given by the labeling on the edges. 
    \(Q\) is the associated quiver.}
    \label{fig:G(2,7) and Q}
\end{figure}

\begin{example}\label{ex:path algebra of contracted G(2,7)}
A framed full DAG \(G\) and its associated quiver \(Q\) are shown in Figure~\ref{fig:G(2,7) and Q}. 
The ideal \(I\) of \(KQ\) is generated by the relations \(\alpha_2\alpha_3,\, \alpha_3\alpha_1\). 
So the path algebra with relations for $G$ is 
\[
\Lambda(G)=\mathrm{span}_K\{\varepsilon_1,\varepsilon_2, \varepsilon_3, \alpha_1, \alpha_2, \alpha_3, \alpha_1\alpha_2\}\, ,
\]
where \(\varepsilon_i\) denotes the constant path at vertex \(i\).
\end{example}

\begin{proposition}
    Let $G$ be a full DAG with a fixed ample framing, then the algebra $\Lambda(G)$ is gentle.
\end{proposition}

\begin{proof}
    It suffices to check that the algebra $\Lambda(G)=KQ/I$ satisfies the conditions (a)--(d) of Definition~\ref{def:gentle}.  By construction, the ideal $I$ is generated by paths of length two, so $I$ is a monomial ideal and satisfies condition (d).  Since $G$ is full, there are at most two arrows starting and ending at every vertex of $Q$, which implies condition (a).  Lastly, conditions (b) and (c) are satisfied due to the properties of weights around every inner vertex of $G$.  
\end{proof}

\begin{definition}
\label{def:gentle-indecomposible}

For $G$ a full DAG, let the set of \emph{shifted projective $\Lambda(G)$-modules} be $\{ P_i[1] : i \text{ is a vertex in } Q\}$ and let $\text{ind}\,\Lambda$ denote the set of \emph{indecomposable $\Lambda(G)$-modules} up to isomorphisms.
We further define
\[
\mathcal{T}(\Lambda(G)):=\text{ind}\,\Lambda\cup \{P_i[1] : i\in Q_0\}\, .
\]
\end{definition}

The set $\mathcal{T}(\Lambda(G))$ can be identified with a set of complexes in the derived category of the module category of $\Lambda(G)$.
Here shifted projectives can be thought of as complexes of projective modules concentrated in degree 1, while $\text{ind}\,\Lambda$ can be viewed as complexes of $\Lambda(G)$-modules concentrated in degree 0. 

\begin{theorem}\label{routes-modules}
Let $G$ be a full DAG with a fixed ample framing, and let $\calR(G)$ denote the set of non-exceptional routes in $G$.
There is a bijection
\[\phi: \calR(G) \longrightarrow \mathcal{T}(\Lambda(G))\, .\] 
\end{theorem}
\begin{proof}
First define a map $\phi: \calR(G)\rightarrow \mathcal{T}(\Lambda(G))$.
Let $R=(e_1,\ldots, e_k)$ be a non-exceptional route in $G$.
If 
\[\omega(R) = (\underbrace{1,\ldots,1=\omega(e_r)}_a, \underbrace{\omega(e_{r+1})=2,\ldots,2}_b)
\] with $a,b\geq1$, then define $\phi(R)$ to be the shifted projective $P_i[1]$ where $i$ is the head of $e_r$ and tail of $e_{r+1}$.
If 
\[\omega(R) = (\underbrace{1,\ldots,1}_a, 2=\omega(e_i), \omega(e_{i+1}), \ldots, \omega(e_{j-1}), \omega(e_j)=1, \underbrace{2,\ldots,2}_b)
\] 
with $a,b\geq0$, then associate to $R$ the string $w=e_{i+1}^{t_{i+1}}\dots e_{j-1}^{t_{j-1}}$ in $Q$, where $t_r=1$ if $\omega(e_r)=1$ and $t_{r}=-1$ if $\omega(e_r)=2$ for $r \in \{ i+1, \dots, j-1\}$.  Note that if $j=i+1$ then in $Q$ the arrows $e_i, e_{i+1}$ start at the same vertex $s(e_i)$, and we set $w=\varepsilon_{s(e_i)}$ to be the constant path at this vertex. 
Let $M(w)$ denote the corresponding indecomposable $\Lambda(G)$-module.
Note that the module $M(w)$ is well-defined because a route in $G$ does not contain incident edges $e_r$ and $e_{r+1}$ such that the path $e_r e_{r+1}$ in $Q$ belongs to the ideal $I$.

Now consider a map $\psi: \mathcal{T}(\Lambda(G))\rightarrow \calR(G)$ such that
$\psi(P_i[1])$ is the unique route in $G$ that passes through vertex $i$ whose edges preceeding $i$ have weight $1$ and edges following $i$ have weight $2$.  If $M(w)$ is an indecomposable string $\Lambda(G)$-module with $w=e_{i+1}^{t_{i+1}}\ldots e_{j-1}^{t_{j-1}}$ and $j-1\geq i+1$, then the string does not pass through any relations of $Q$; hence $e_{i+1}, \ldots, e_{j-1}$ is a path in $G$.  Define $\psi(M(w))$ to be the unique route in $G$ with weight $(1,\ldots, 1,2= \omega(e_i),\omega(e_{i+1}),\ldots,\omega(e_{j-1}), \omega(e_j)=1,2\ldots, 2)$.  Lastly, if $M(w) = M(\varepsilon_r)$ then define $\psi(M(w))$ to be the unique route in $G$ with weight $(1,\ldots, 1,2= \omega(e_i), \omega(e_{i+1})=1,2\ldots, 2)$ such that the tail of $e_i$ in $G$ is the vertex $r$. 

The maps $\phi$ and $\psi$ are inverses, and this proves the theorem.
\end{proof}

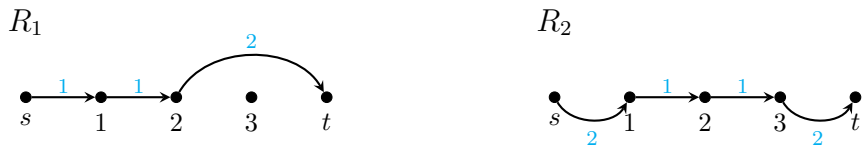
\begin{figure}
\begin{center}
    \begin{tikzpicture}
    
    \begin{scope}[scale=1, xshift=0, yshift=0]
	\vertex[fill,label=below:\footnotesize{$s$}](as) at (1,0) {};
	\vertex[fill,label=below:\footnotesize{$1$}](a1) at (2,0) {};
	\vertex[fill,label=below:\footnotesize{$2$}](a2) at (3,0) {};
	\vertex[fill,label=below:\footnotesize{$3$}](a3) at (4,0) {};
	\vertex[fill,label=below:\footnotesize{$t$}](at) at (5,0) {};
	
	\draw[-stealth, thick] (as)--(a1);
	\draw[-stealth, thick] (a1)--(a2);

	\draw[-stealth, thick] (a2) to[out=60,in=120] (at);

	
    \node[] at (1.5,0.15) {\textcolor{cyan}{\tiny$1$}}; 
    \node[] at (2.5,0.15) {\textcolor{cyan}{\tiny$1$}}; 
    
    \node[] at (4,0.75) {\textcolor{cyan}{\tiny$2$}}; 
    
    
    \node[] at (1,1) {$R_1$};

    \end{scope}

    \begin{scope}[scale=1, xshift=200, yshift=0]
	\vertex[fill,label=below:\footnotesize{$s$}](as) at (1,0) {};
	\vertex[fill,label=below:\footnotesize{$1$}](a1) at (2,0) {};
	\vertex[fill,label=below:\footnotesize{$2$}](a2) at (3,0) {};
	\vertex[fill,label=below:\footnotesize{$3$}](a3) at (4,0) {};
	\vertex[fill,label=below:\footnotesize{$t$}](at) at (5,0) {};
	
	\draw[-stealth, thick] (a1)--(a2);
 	\draw[-stealth, thick] (a2)--(a3);


 	\draw[-stealth, thick] (as) to[out=-60,in=-120] (a1);
 	\draw[-stealth, thick] (a3) to[out=-60,in=-120] (at);
	
     \node[] at (2.5,0.15) {\textcolor{cyan}{\tiny$1$}}; 
     \node[] at (3.5,0.15) {\textcolor{cyan}{\tiny$1$}}; 
    
    
     \node[] at (1.5,-0.55) {\textcolor{cyan}{\tiny$2$}}; 
     \node[] at (4.5,-0.55) {\textcolor{cyan}{\tiny$2$}}; 
    
    \node[] at (1,1) {$R_2$};

    \end{scope}    

    \end{tikzpicture}    
\end{center}
    \caption{Two non-exceptional routes in the DAG $G$ from Figure~\ref{fig:G(2,7) and Q}, with \(R_1\) corresponding to the shifted projective module \(P_2[1]\), and \(R_2\) corresponding to the indecomposable \(\Lambda(G)\)-module \(M(\alpha_1)\).}
    \label{fig:non-exceptionals of G(2,7)}
\end{figure}

Figure~\ref{fig:non-exceptionals of G(2,7)} provides examples of the correspondence between routes and modules as given in Theorem~\ref{routes-modules}.
This correspondence gives rise to the following corollary which determines the representation type of the path algebra for full DAGs.

\begin{corollary}\label{cor:finite}
Every indecomposable $\Lambda(G)$-module has dimension at most one at every vertex.  In particular, the algebra $\Lambda(G)$ is of finite representation type.
\end{corollary}
\begin{proof}
This follows directly from the bijection in Theorem~\ref{routes-modules}. Since $G$ is acyclic then no route passes through the same vertex twice, so every indecomposable $\Lambda(G)$-module has dimension at most one at every vertex. 
\end{proof}


\begin{remark}
The converse of the above corollary does not hold; in particular, not every gentle algebra of finite representation type comes from a full DAG.  For example, consider an algebra given by the following quiver 
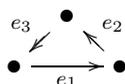
\[\xymatrix@C=7pt@R=7pt{ &\bullet\ar[dl]_{e_3}\\
\bullet \ar[rr]_{e_1} && \bullet \ar[ul]_{e_2}}
\]
with relations $e_1e_2=e_2e_3=e_3e_1=0$.
If such an algebra were to come from a full DAG, then $\omega(e_1)\not=\omega(e_2)$, $\omega(e_2)\not=\omega(e_3)$, and $\omega(e_3)\not=\omega(e_1)$, which is not possible. 
\end{remark}

We will be interested in a special property of objects in $\mathcal{T}(\Lambda(G))$ called $\tau$-\emph{rigidity}.  It was defined in purely homological terms and studied for general finite dimensional algebras in~\cite{AIR}.  In the case of gentle algebras, $\tau$-rigid modules were studied in~\cite{BDMTY,PPP}, which allows us to translate the definition of $\tau$-rigidity into purely combinatorial terms.
Next, we recall the relevant construction.  

Let $\Lambda=KQ/I$ be a gentle algebra.  We construct an extended algebra $\widehat{\Lambda}=K\widehat{Q}/ \widehat{I}$, called a \emph{blossoming algebra}, as follows.
The quiver $\widehat{Q}$ is obtained from $Q$ by adding sources and sinks such that each new vertex is incident to a single arrow and every vertex $i$ of $Q$ has two arrows in $\widehat{Q}$ leaving $i$ and two arrows in $\widehat{Q}$ starting at $i$.
Moreover, we impose additional relations on $\widehat{Q}$ given by paths of length two to get an ideal $\widehat{I} \supset I$ such that the resulting algebra $\widehat{\Lambda}$ becomes gentle.
Note that $\widehat{\Lambda}$ is unique up to permuting the set of sinks and permuting the set of sources. 
See Figure~\ref{fig:Q hat} for an example.

\begin{figure}
\begin{center}
    
    \begin{tikzpicture}
    
    \begin{scope}[scale=0.9, xshift=0, yshift=0]
    \vertex[fill,label=below:\footnotesize{$a_1$}](as1) at (1,1) {};
    \vertex[fill,label=below:\footnotesize{$a_2$}](as2) at (1,0) {};
    \vertex[fill,label=below:\footnotesize{$a_3$}](as3) at (1,-1) {};    
	\vertex[fill,label=below:\footnotesize{$1$}](a1) at (2,0) {};
	\vertex[fill,label=below:\footnotesize{$2$}](a2) at (3,0) {};
	\vertex[fill,label=below:\footnotesize{$3$}](a3) at (4,0) {};
    \vertex[fill,label=below:\footnotesize{$b_1$}](at1) at (5,1) {};
    \vertex[fill,label=below:\footnotesize{$b_2$}](at2) at (5,0) {};
    \vertex[fill,label=below:\footnotesize{$b_3$}](at3) at (5,-1) {};    
	
	\draw[-stealth, thick] (a1)--(a2);
	\draw[-stealth, thick] (a2)--(a3);
	\draw[-stealth, thick] (a3) to[out=-120,in=-55] (a1);
	
	\draw[-stealth, thick] (as2)--(a1);
	\draw[-stealth, thick] (a3)--(at2);
	
	\draw[-stealth, thick] (a2) to[out=120,in=10] (as1);
	\draw[-stealth, thick] (at1) to[out=170,in=60] (a2);
	\draw[-stealth, thick] (a1)--(as3);
	\draw[-stealth, thick] (at3)--(a3);
    \node[] at (1,1) {};    
    \end{scope}
    
    \end{tikzpicture}
    
\end{center}

    \caption{The quiver \(\widehat{Q}\) obtained from the quiver \(Q\) in Figure~\ref{fig:G(2,7) and Q} by adding sources and sinks.}
    \label{fig:Q hat}
\end{figure}
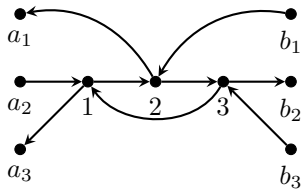

Next, we show how to extend objects in $\mathcal{T}(\Lambda)$ to certain modules over the corresponding blossoming algebra. 
\begin{definition}\label{def:string}
We define a map 
\[
\xi: \mathcal{T}(\Lambda) \to \text{ind}\,\widehat{\Lambda},
    M(w)\mapsto M(\widehat{w})
\]
as follows.
Let $w$ be a string in $\Lambda$.  
Since $\Lambda$ is gentle, the string $\widehat{w}$ that we construct next will be well-defined and unique.
If the length of $w$ is at least one, then define $\widehat{w}$ to be the string in $\widehat{\Lambda}$ obtained by extending $w$ at the start of $w$ by an inverse arrow, followed by adding as many direct arrows as possible until a source vertex is reached, and then similarly extending $w$ at the end by an arrow, followed by adding as many inverse arrows as possible until a sink vertex is reached.
Otherwise, if the length of $w$ is zero, i.e., $w$ is a constant path $\varepsilon_i$ at vertex $i$, then there are two arrows pointing away from $i$ in the blossoming quiver.  
In this case, define $\widehat{w}$ to be the string in $\widehat{\Lambda}$ obtained by extending $w$ by these two arrows followed by adding as many directed arrows at the start as possible until a source vertex is reached and as many inverse arrows at the end as possible until a sink vertex is reached.  
Finally, consider a shifted projective $P_i[1]=M(i[1])$, where we think of $i[1]$ as a ``shifted string" at vertex $i$.
Then we define the corresponding string $\widehat{i[1]}$ in the blossoming algebra to be the maximal string supported at vertex $i$ such that the path preceding $i$ consists of arrows while the path following $i$ consists of inverse arrows.
The map on the strings induces an inclusion on the modules, which sends a module $M(w)$ to $M(\widehat{w})$.
\end{definition}

\begin{example}
Consider $G=G(2,7)$ defined in Definition~\ref{def:gkn}, and let $Q$ be its associated quiver. Building on Example~\ref{ex:path algebra of contracted G(2,7)} and Figure~\ref{fig:Q hat}, we demonstrate the extension of strings in the blossoming algebra as described above. Figure~\ref{fig:Q hat2} depicts $\widehat{Q}$ along with strings $w_1=\alpha_2$, $w_2=\varepsilon_2$, and their extensions. In regard to uniqueness of the extension procedure, note that if we instead chose to extend $w_1$ by $\alpha_3$ instead of $\alpha_5$, the resulting ``extended string'' would contain $\alpha_2\alpha_3$ which is a zero relation in $\widehat{\Lambda}$.
\end{example}

It is easy to see that the image of the map in Definition~\ref{def:string} consists of all string modules $M(\widehat{w})$ such that $\widehat{w}$ is a maximal undirected string in $\widehat{\Lambda}$, i.e., it is an undirected string that starts in a source and ends in a sink.  
Moreover, it then restricts to a bijection between $\mathcal{T}(\Lambda)$ and its image. 

\begin{figure}
\begin{center}
\begin{tikzpicture}

    \begin{scope}[scale=0.9, xshift=0, yshift=0]
    \vertex[fill,label=left:\footnotesize{$a_1$}](as1) at (1,1) {};
    \vertex[fill,label=left:\footnotesize{$a_2$}](as2) at (1,0) {};
    \vertex[fill,label=left:\footnotesize{$a_3$}](as3) at (1,-1) {};    
	\vertex[fill,label=below:\footnotesize{$1$}](a1) at (2,0) {};
	\vertex[fill,label=below:\footnotesize{$2$}](a2) at (3,0) {};
	\vertex[fill,label=below:\footnotesize{$3$}](a3) at (4,0) {};
    \vertex[fill,label=right:\footnotesize{$b_1$}](at1) at (5,1) {};
    \vertex[fill,label=right:\footnotesize{$b_2$}](at2) at (5,0) {};
    \vertex[fill,label=right:\footnotesize{$b_3$}](at3) at (5,-1) {};    
	
	\draw[-stealth, thick] (a1)--(a2);
	\draw[-stealth, thick] (a2)--(a3);
	\draw[-stealth, thick] (a3) to[out=-120,in=-55] (a1);
	
	\draw[-stealth, thick] (as2)--(a1);
	\draw[-stealth, thick] (a3)--(at2);
	
	\draw[-stealth, thick] (a2) to[out=120,in=10] (as1);
	\draw[-stealth, thick] (at1) to[out=170,in=60] (a2);
	\draw[-stealth, thick] (a1)--(as3);
	\draw[-stealth, thick] (at3)--(a3);
    \node[] at (1,1) {};    
    
    \node[] (alpha4) at (2.2,1.1) {\tiny $\alpha_4$};
    \node[] (alpha1) at (2.4,0.2) {\tiny $\alpha_1$}; \node[] (alpha2) at (3.6,0.2) {\tiny $\alpha_2$}; \node[] (alpha3) at (3,-0.8) {\tiny $\alpha_3$};
    \node[] (alpha5) at (4.5,0.2) {\tiny $\alpha_5$};
    \node[] (alpha6) at (4.4,-0.8) {\tiny $\alpha_6$};
    \end{scope}
    
    \begin{scope}[scale=0.9, xshift=-150, yshift=-80]
    \vertex[fill,label=below:\footnotesize{$2$}](a2) at (3,0) {};
	\vertex[fill,label=below:\footnotesize{$3$}](a3) at (4,0) {};
    
	\draw[-stealth, thick] (a2)--(a3);
    \node[] at (2,0) {$w_1:$};    
    
    \node[] (alpha2) at (3.5,0.2) {\tiny $\alpha_2$}; \end{scope}
    
    \begin{scope}[scale=0.9, xshift=-92, yshift=-150]
    \vertex[fill,label=left:\footnotesize{$a_1$}](as1) at (1,1) {};
    \vertex[fill,label=below:\footnotesize{$2$}](a2) at (3,0) {};
	\vertex[fill,label=below:\footnotesize{$3$}](a3) at (4,0) {};
    \vertex[fill,label=right:\footnotesize{$b_2$}](at2) at (5,0) {};
    
	\draw[-stealth, thick] (a2)--(a3);
	\draw[-stealth, thick] (a3)--(at2);
	\draw[stealth-, thick] (a2) to[out=120,in=10] (as1);
	
    \node[] at (0,0.5) {$\widehat{w}_1:$};

    \node[] (alpha4) at (2.3,1.1) {\tiny $\alpha_4^{-1}$};
    \node[] (alpha2) at (3.5,0.2) {\tiny $\alpha_2$}; \node[] (alpha5) at (4.5,0.2) {\tiny $\alpha_5$};
    \end{scope}
    
    \begin{scope}[scale=0.9, xshift=92, yshift=-80]
    \vertex[fill,label=below:\footnotesize{$2$}](a2) at (3,0) {};

    \node[] at (2,0) {$w_2:$};    
    
    \end{scope}
    
    \begin{scope}[scale=0.9, xshift=150, yshift=-150]
    \vertex[fill,label=left:\footnotesize{$a_1$}](as1) at (1,1) {};
    \vertex[fill,label=below:\footnotesize{$2$}](a2) at (3,0) {};
	\vertex[fill,label=below:\footnotesize{$3$}](a3) at (4,0) {};
    \vertex[fill,label=right:\footnotesize{$b_3$}](at3) at (5,-1) {};  
    
	\draw[-stealth, thick] (a2)--(a3);
	\draw[stealth-, thick] (a2) to[out=120,in=10] (as1);
    \draw[stealth-, thick] (at3)--(a3);

    \node[] at (0,0.5) {$\widehat{w}_2:$};

    \node[] (alpha4) at (2.3,1.1) {\tiny $\alpha_4^{-1}$};
    \node[] (alpha2) at (3.5,0.2) {\tiny $\alpha_2$}; 
    \node[] (alpha6) at (4.8,-0.3) {\tiny $\alpha_6^{-1}$};
    \end{scope}
    
\end{tikzpicture}
\end{center}
    \caption{The quiver $\widehat{Q}$ from Figure~\ref{fig:Q hat} with edge labels is pictured above. The string $w_1=\alpha_2$ and its corresponding extension $\widehat{w}_1$ are depicted on the left. The string $w_2=\varepsilon_2$ and its corresponding extension $\widehat{w}_2$ are shown on the right.}
    \label{fig:Q hat2}
\end{figure}
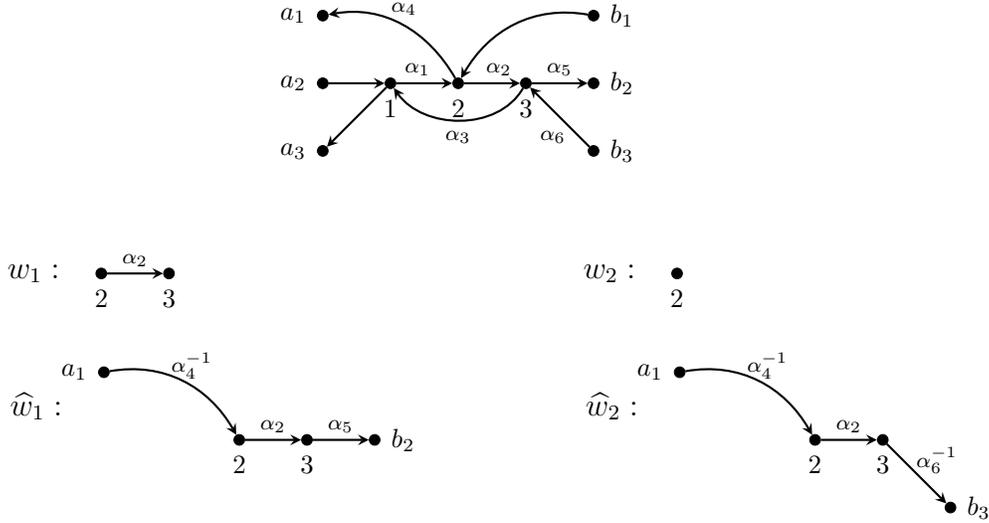

This enables us to reformulate the original definition of $\tau$-rigidity given in \cite{AIR} in the following combinatorial terms. 

\begin{definition}\cite[Theorem 4.3]{BDMTY}\cite[Theorem 2.46]{PPP}
Let $\Lambda$ be a gentle algebra, and let $M(w), M(w')$ be two objects in $\mathcal{T}(\Lambda)$.  
\begin{itemize}
    \item[(a)]  The pair $M(w), M(w')$ is said to be \emph{$\tau$-rigid} if $\widehat{w}, \widehat{w}'$ do not contain a common substring $\sigma$ such that the arrows of $\widehat{w}$ incident to $\sigma$ are both outgoing while the arrows of $\widehat{w}'$ incident to $\sigma$ are both incoming.  
    \item[(b)] The object $M(w)$ is said to be \emph{$\tau$-rigid} if the pair $M(w), M(w)$ is $\tau$-rigid. 
    \item[(c)] A collection of objects in $\mathcal{T}(\Lambda)$ is said to be \emph{support $\tau$-tilting} if it is a maximal collection of pairwise $\tau$-rigid objects. 
\end{itemize}
\end{definition}

Note that for an arbitrary gentle algebra $\Lambda$, condition (b) above does not always hold.
However, in our setting an algebra obtained from a DAG is of finite representation type such that every string passes through a vertex of $Q$ at most once, and hence (b) holds for all strings $w$.

By using the bijections between routes in $G$, objects in $\mathcal{T}(G)=\mathcal{T}(\Lambda(G))$, and maximal undirected string modules in $\widehat{\Lambda}$ discussed above, we can translate the notion of $\tau$-rigidity in the module category of $\Lambda(G)$ in terms of routes in $G$ by passing through the blossoming algebra.  
Let $\Lambda(G) = KQ/I$ be a gentle algebra coming from a full DAG $G$, and let $\widehat{\Lambda}(G)$ be the corresponding blossoming algebra.
Applying the reverse construction of Definition~\ref{def:quiver}, starting from $\widehat{\Lambda}(G)$ we can obtain a full DAG $\widehat{G}$ with a framing such that $\widehat{G}$ has the same vertices as $\widehat{Q}$ and whose arrows come from relations in the ideal $\widehat{I}$. 
Note that $G$ and $\widehat{G}$ agree on the interior vertices, and moreover $G$ together with the framing can be obtained from $\widehat{G}$ by gluing certain sources together and/or certain sinks together. In particular, $\Lambda(G)=\Lambda(\widehat{G})$ and we also obtain equality for the corresponding flow polytopes, that is  $\mathcal{F}_1(G) = \mathcal{F}_1(\widehat{G})$.  

With this notation consider the following statement, where $\phi$ is the bijection from Theorem~\ref{routes-modules}.

\begin{lemma}\label{compatibility}
Let $G$ be a full DAG with a fixed ample framing, then the following properties hold. 
\begin{itemize}
\item[(a)] Every indecomposable object in $\mathcal{T}(G)$ is $\tau$-rigid. 
\item[(b)]Two routes $R, R'$ in $G$ are coherent if and only if the corresponding objects $\phi(R), \phi(R')$ of $\mathcal{T}(G)$ are $\tau$-rigid. 
\end{itemize}
\end{lemma}

\begin{proof}
Part (a) follows directly from the definition of $\tau$-rigidity and Corollary~\ref{cor:finite}.  

To show part (b), let $R, R'$ be two routes in $G$ and let $\phi(R)=M(w), \phi(R')=M(w')$ denote the corresponding objects in $\mathcal{T}(G)$, as described in the proof of Theorem~\ref{routes-modules}.  Now let $M(\widehat{w}), M(\widehat{w}')$ denote the associated modules over the blossoming algebra $\widehat{\Lambda}$. Each $\widehat{w}, \widehat{w}'$ is an undirected walk in $\widehat{Q}$ from a source to a sink, which gives routes $\widehat{R}, \widehat{R}'$ in $\widehat{G}$.  We see that $\widehat{R}, \widehat{R}'$ are coherent in $\widehat{G}$ if and only if the routes $R, R'$ are coherent in $G$. Therefore, it suffices to show that $M(w), M(w')$ are $\tau$-rigid if and only if the routes $\widehat{R}, \widehat{R}'$ are coherent. 

By definition $M(w), M(w')$ are $\tau$-rigid whenever there does not exist a common substring $\sigma$ of $\widehat{w}, \widehat{w}'$ such that the arrows of $\widehat{w}$ incident to $\sigma$ are both outgoing while the arrows of $\widehat{w}'$ incident to $\sigma$ are both incoming.  This means that there does not exist a common subpath $R_\sigma$ of $\widehat{R}, \widehat{R}'$ such that the arrows in $\widehat{R}$ ending and starting at $R_\sigma$ have weights 2 and 1 respectively while the arrows in $\widehat{R}'$ ending and starting at $R_\sigma$ have weights 1 and 2 respectively.  This means that the routes $\widehat{R}, \widehat{R}'$ are coherent, see Definition~\ref{def:coherent}.  This shows that if $M(w), M(w')$ are $\tau$-rigid then the routes $\widehat{R}, \widehat{R}'$ are coherent.  The converse follows in the same way. 
\end{proof}

\begin{theorem}\label{maximal_cliques}
Let $G$ be a full DAG with a fixed ample framing.
The bijection $\phi$ from Theorem~\ref{routes-modules} induces a bijection 
\[\Phi: \{R_1, \dots, R_n\} \mapsto \{\phi(R_1), \dots, \phi(R_n)\}\] 
 between the set of maximal cliques of $G$ and the set of support $\tau$-tilting $\Lambda(G)$-modules.  
\end{theorem}

\begin{proof}
A maximal clique of $G$ is a maximal collection of pairwise coherent routes, and similarly a support $\tau$-tilting module is a collection of pairwise $\tau$-rigid modules.  Thus, the result follows from Theorem~\ref{routes-modules} and Lemma~\ref{compatibility}.
\end{proof}

Let $\text{st}(\Lambda)$ denote the set of support $\tau$-tilting modules over an algebra $\Lambda$.
It follows from \cite{AIR} that $\text{st}(\Lambda)$ for any finite dimensional algebra over an algebraically closed field has a poset structure.
This poset structure is called the \emph{$\tau$-tilting poset for $\Lambda$}.
In the case of gentle algebras this poset was described combinatorially in \cite{PPP}, and we will recall some relevant properties in the next section.
For now, we note that the Hasse diagram of the $\tau$-tilting poset is the dual graph of a simplicial complex called the \emph{$\tau$-tilting complex}.
The following corollary now follows from the bijection $\Phi$ in Theorem~\ref{maximal_cliques}. 

\begin{corollary} Let $G$ be a full DAG with an ample framing $F$. The dual graph of the triangulation $\mathrm{DKK}(G,F)$ is the Hasse diagram of the $\tau$-tilting poset on $\mathrm{st}(\Lambda(G))$.
Furthermore, the $\tau$-tilting complex is isomorphic to $\mathrm{DKK}(G,F)$.
\end{corollary}

\begin{proof}
    Suppose that two maximal cliques $\Delta, \Delta'$ of $G$ correspond to an edge in the dual graph of the triangulation $\mathrm{DKK}(G,F)$.  Then they differ by a single route, that is $\Delta=\Delta'\setminus \{R'\}\cup \{R\}$ for distinct routes $R, R'$ of $G$. By definition of $\Phi$, the corresponding support $\tau$-tilting modules $\Phi(\Delta), \Phi(\Delta')$ of $\Lambda(G)$ differ by a single indecomposable module.  This means that they are connected by an edge in the poset $\mathrm{st}(\Lambda(G))$.  The converse follows in the same way, since $\Phi$ is a bijection. This shows the first part of the statement. The second part can be deduced analogously by naturally extending $\Phi$ to a bijection between cliques of $G$  and collections of pairwise $\tau$-rigid objects of $\Lambda(G)$.
\end{proof}

The following theorem will be critical in our work.

\begin{theorem}\cite[Theorem 5.4]{DIJ}
\label{thm:tau-shellable-rep}
Let $\Lambda$ be an algebra with a finite $\tau$-tilting poset.
Then the $\tau$-tilting complex is shellable.
Moreover, every linear extension of the $\tau$-tilting poset yields a shelling order.
\end{theorem}

In particular, combining Theorems \ref{thm:tau-shellable-rep} and \ref{maximal_cliques}, we obtain the following result. 

\begin{corollary}
    
    \label{thm:tau-shellable}
    For the poset structure of maximal cliques of a full DAG given by the correspondence in Theorem~\ref{maximal_cliques}, every linear extension yields a shelling of the DKK triangulation.
\end{corollary}

\begin{figure}
\centering
\begin{tabular}{cc}
\hline
Flow polytopes & Gentle algebras\\
\hline
$G\backslash \{s,t\}$ 
    & $Q_0$\\
edges of $G$ not incident to $s$ or $t$ 
    & $Q_1$\\
non-exceptional routes $\mathcal{R}(G)$
    & $\mathcal{T}(\Lambda(G))=\text{ind}\,\Lambda\cup \{P_i[1] : i\in Q_0\}$\\
a pair of coherent routes 
    & a $\tau$-rigid pair of modules\\
a maximal clique of routes
    & support $\tau$-tilting modules\\
the triangulation $\mathrm{DKK}(G,F)$    
    & the $\tau$-tilting complex\\
dual graph of the triangulation $\mathrm{DKK}(G,F)$    
    & Hasse diagram of $\tau$-tilting poset on $\mathrm{st}(\Lambda(G))$\\
\hline
\end{tabular}
    \caption{A summary of equivalent terminology between DKK triangulations of amply framed flow polytopes and representations of gentle algebras.}
    \label{fig:dictionary}
\end{figure}

\section{Gorenstein Flow Polytopes}\label{sec:gorenstein}

In the case of gentle algebras, the partial order in the $\tau$-tilting poset can be reformulated in combinatorial terms as proved in \cite[Theorem 6.2]{BDMTY} for $\tau$-tilting finite algebras or \cite[Theorem 2.46]{PPP} in full generality.  The precise formulation of the partial order between two adjacent support $\tau$-tilting modules can be found, for example, in \cite[Proposition 2.33]{PPP}, which we present below using the terminology of routes. 
In Definition~\ref{def:order}, we describe the partial order using the bijection between non-exceptional routes of $G$ and maximal undirected strings in the corresponding blossoming algebra described in the previous section, see Definition~\ref{def:string}.  

\begin{definition}\label{def:order}
Let $G$ be a full DAG with ample framing $F$ and associated triangulation $\DKK(G,F)$.
Let $\Delta_1=\Delta\cup \{R_1\}$ and $\Delta_2=\Delta\cup \{R_2\}$ be adjacent maximal cliques in $\DKK(G,F)$.
Then the collection of vertices and edges in 
$R_1\cap R_2$ has a unique connected component $w$ such that the edge entering $w$ in $R_1$ is labeled $2$ and the edge exiting $w$ in $R_1$ is labeled $1$, and vice versa for $R_2$.
We define an ordering $\prec'$ on pairs of adjacent facets of $\DKK(G,F)$ where $\Delta_2\prec' \Delta_1$ in the case above, and we extend $\prec'$ to a partial order $\prec$ by taking the transitive closure of $\prec'$.
\end{definition}

\begin{example}
Let $G$ be the full contraction of $G(2,7)$ as shown in Figure~\ref{fig:G(2,7) and Q}.
Let $\Delta_1$, $\Delta_2$, $\Delta_1'$, and $\Delta_2'$ be facets in $\DKK(G,F)$, where $F$ is the length framing.
If $\Delta_1 = \Delta \cup \{R_1\}$ and $\Delta_2 = \Delta \cup \{R_2\}$, with $R_1$ and $R_2$ as in Figure~\ref{fig:coverRelation}, then $w = (1,2)$, and we have that $\Delta_2 \prec \Delta_1$. 
If $\Delta_1' = \Delta' \cup \{R_1'\}$ and $\Delta_2' = \Delta' \cup \{R_2'\}$, with $R_1'$ and $R_2'$ as in Figure~\ref{fig:coverRelation}, then although $R_1'\cap R_2' = \{1,3\}$, only vertex $3$ has incoming edge labeled $2$ and outgoing edge labeled $1$ for $R_1'$ and vice versa for $R_2'$. 
Thus $w = \{3\}$, and we see that $\Delta_2'\prec \Delta_1'$.
Figure~\ref{fig:G(2,7)poset} shows the $\tau$-tiling poset with partial order $\prec$ for the framed DAG $G$.
A smaller and more detailed example is given in Figure~\ref{fig:G(2,6) and poset} for a full contraction of $G(2,6)$ with the length framing.
\end{example}

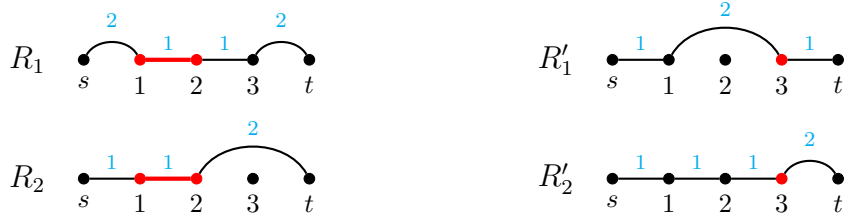
\begin{figure}
    \centering
    \begin{tikzpicture}
    \begin{scope}[xshift=-30, scale=1.5]
    \begin{scope}[scale=0.5, xshift=0, yshift=0]

   \node[] (R2) at (0,0) {$R_2$};

	\vertex[fill,label=below:\footnotesize{$s$}](as) at (1,0) {};
	\vertex[fill,
	color=red,label=below:\footnotesize{$1$}](a1) at (2,0) {};
	\vertex[fill,
	color=red,label=below:\footnotesize{$2$}](a2) at (3,0) {};
	\vertex[fill,label=below:\footnotesize{$3$}](a3) at (4,0) {};
	\vertex[fill,label=below:\footnotesize{$t$}](at) at (5,0) {};
	
	\draw[thick] (as)--(a1);
	\node[] at (1.5,0.3) {\textcolor{cyan}{\tiny$1$}}; 
	
	\draw[ultra thick, color=red] (a1)--(a2);
	\node[] at (2.5,0.3) {\textcolor{cyan}{\tiny$1$}};  
	


	

	\draw[thick] (a2) to[out=60,in=120] (at);
	\node[] at (4,0.9) {\textcolor{cyan}{\tiny$2$}}; 
	
	
    
    \end{scope}
    
    \begin{scope}[scale=0.5, xshift=0, yshift=60]
    
    \node[] (R1) at (0,0) {$R_1$};
    
	\vertex[fill,label=below:\footnotesize{$s$}](as) at (1,0) {};
	\vertex[fill,color=red,label=below:\footnotesize{$1$}](a1) at (2,0) {};
	\vertex[fill,color=red,label=below:\footnotesize{$2$}](a2) at (3,0) {};
	\vertex[fill,label=below:\footnotesize{$3$}](a3) at (4,0) {};
	\vertex[fill,label=below:\footnotesize{$t$}](at) at (5,0) {};
	
	
	\draw[ultra thick, color=red] (a1)--(a2);
	\node[] at (2.5,0.3) {\textcolor{cyan}{\tiny$1$}};  
	
	\draw[thick] (a2)--(a3);
	\node[] at (3.5,0.3) {\textcolor{cyan}{\tiny$1$}};  


	

	
	\draw[thick] (as) to[out=60,in=120] (a1);
	\node[] at (1.5,0.7) {\textcolor{cyan}{\tiny$2$}}; 
	
    
	\draw[thick] (a3) to[out=60,in=120] (at);
	\node[] at (4.5,0.7) {\textcolor{cyan}{\tiny $2$}}; 
    \end{scope}
    \end{scope}
    
    \begin{scope}[xshift=170, scale=1.5]
    \begin{scope}[scale=0.5, xshift=0, yshift=60]

    \node[] (R2) at (0,0) {$R_1'$};

	\vertex[fill,label=below:\footnotesize{$s$}](as) at (1,0) {};
	\vertex[fill,label=below:\footnotesize{$1$}](a1) at (2,0) {};
	\vertex[fill,label=below:\footnotesize{$2$}](a2) at (3,0) {};
	\vertex[fill,
	color=red,label=below:\footnotesize{$3$}](a3) at (4,0) {};
	\vertex[fill,label=below:\footnotesize{$t$}](at) at (5,0) {};
	
	\draw[thick] (as)--(a1);
	\node[] at (1.5,0.3) {\textcolor{cyan}{\tiny$1$}}; 
	
	

	\draw[thick] (a3)--(at);
    \node[] at (4.5,0.3) {\textcolor{cyan}{\tiny$1$}}; 

	

	
	
	\draw[thick] (a1) to[out=60,in=120] (a3);
	\node[] at (3, 0.9)  {\textcolor{cyan}{\tiny$2$}}; 
    
    \end{scope}
    
    \begin{scope}[scale=0.5, xshift=0, yshift=0]
    
    \node[] (R1) at (0,0) {$R_2'$};
    
	\vertex[fill,label=below:\footnotesize{$s$}](as) at (1,0) {};
	\vertex[fill,label=below:\footnotesize{$1$}](a1) at (2,0) {};
	\vertex[fill,label=below:\footnotesize{$2$}](a2) at (3,0) {};
	\vertex[fill, color=red,label=below:\footnotesize{$3$}](a3) at (4,0) {};
	\vertex[fill,label=below:\footnotesize{$t$}](at) at (5,0) {};
	
	\draw[thick] (as)--(a1);
	\node[] at (1.5,0.3) {\textcolor{cyan}{\tiny$1$}}; 
	
	\draw[thick] (a1)--(a2);
	\node[] at (2.5,0.3) {\textcolor{cyan}{\tiny$1$}};  
	
	\draw[thick] (a2)--(a3);
	\node[] at (3.5,0.3) {\textcolor{cyan}{\tiny$1$}};  


	

	
	
    
	\draw[thick] (a3) to[out=60,in=120] (at);
	\node[] at (4.5,0.7) {\textcolor{cyan}{\tiny $2$}}; 
    \end{scope}
    \end{scope}
    
    \end{tikzpicture}
    \caption{Two pairs of routes in the graph $G$ from Figure~\ref{fig:G(2,7) and Q}. The connected component $w$ in $R_1\cap R_2$ is the edge $(1,2)$, while in $R_1'\cap R_2'$ the connected component is the vertex $3$.}
    \label{fig:coverRelation}
\end{figure}

\begin{remark}
The uniqueness in Definition~\ref{def:order} is a consequence of a more general result on gentle algebras, which give a constructive way to calculate adjacent support $\tau$-tilting modules. 
In particular, this definition is an application of \cite[Proposition 2.33]{PPP} together with the bijection in Theorem~\ref{maximal_cliques}, which allows us to restate it in terms of routes.  
Thus, Definition~\ref{def:order} may be formulated as follows.  Given a maximal clique $\Delta_1$ containing a non-exceptional route $R_1$, we can write $R_1=uwv$ where $w$ is uniquely determined by $\Delta_1$ and the last edge in $u$ has the opposite label of the first edge in $v$.
Moreover, $\Delta_1$ contains two other routes $p=u'wv$ and $q=uwv'$.
Then Proposition 2.33 says that the clique $\Delta_2$ obtained by exchanging $R_1$ by a new route $R_2$, where $R_2=u'wv'$, is the unique other maximal clique that contains $\Delta$ and such that $R_1, R_2$ are not coherent at $w$.
Note that if there is another connected component $w' \in R_1\cap R_2$ satisfying the conditions of Definition~\ref{def:order}, then $R_1$ and $R_2$ would be incoherent at $w'$, and in particular $w'$ would be contained in $u$ or $v$.
Hence, $w'$ would be contained in $q$ or $p$ respectively.
Then $R_2$ would be incoherent with either $q$ or $p$, which contradicts that $R_2$, $p$, and $q$ belong to a common clique $\Delta_2$.

    Alternatively, the uniqueness was also shown in \cite[Theorem 9.4]{BDMTY} in the special case of gentle algebras with the property that every indecomposable $\tau$-rigid module is a brick.
This condition is automatically satisfied for algebras coming from full DAGs by Corollary~\ref{cor:finite}.
\end{remark}

\begin{figure}
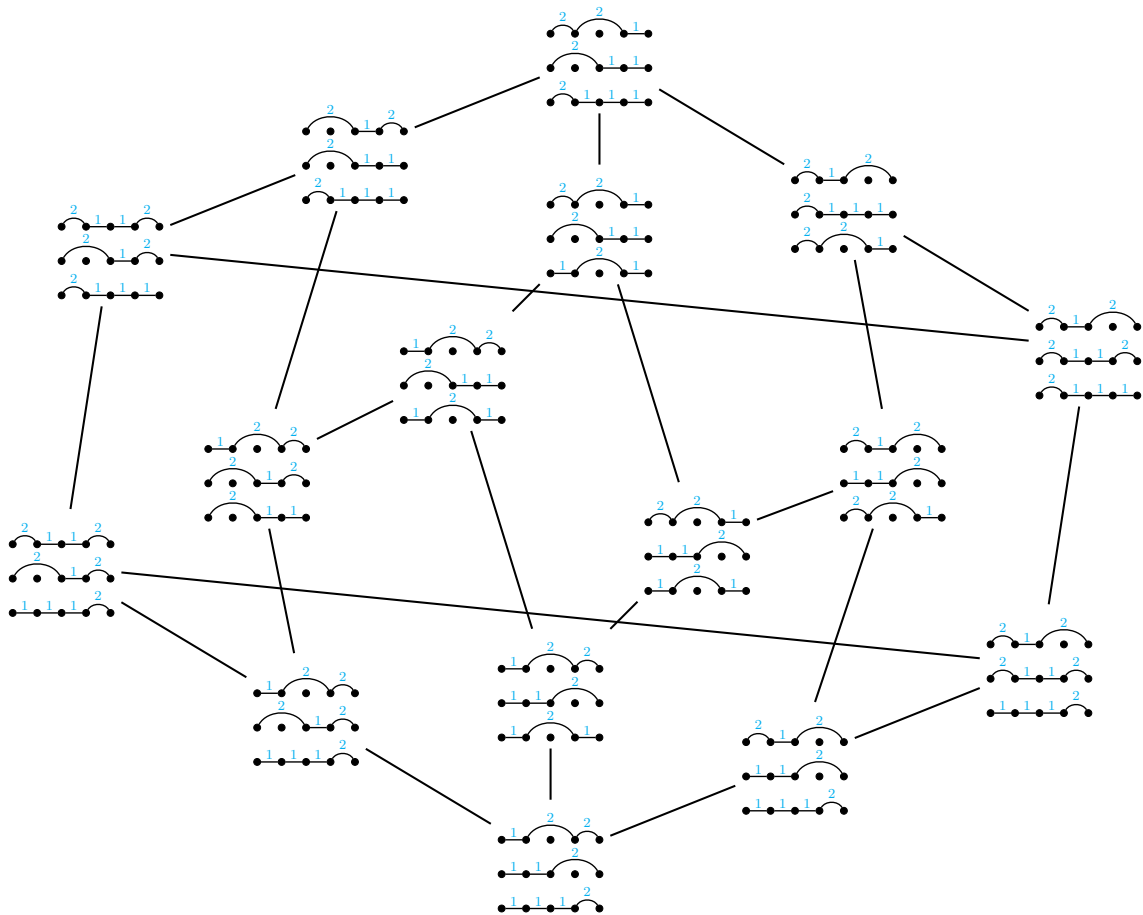

\centering
\scalebox{0.65}{

    };        

\draw[very thick] (16)--(14);
\draw[very thick] (16)--(15);
\draw[very thick] (16)--(12);
\draw[very thick] (15)--(10);
\draw[very thick] (15)--(9);
\draw[very thick] (12)--(11);
\draw[very thick] (14)--(13);
\draw[very thick] (13)--(11);
\draw[very thick] (10)--(4);
\draw[very thick] (9)--(4);
\draw[very thick] (4)--(1);
\draw[very thick] (14)--(8);
\draw[very thick] (12)--(6);
\draw[very thick] (9)--(6);
\draw[very thick] (10)--(8);
\draw[very thick] (1)--(2);
\draw[very thick] (1)--(3);
\draw[very thick] (2)--(6);
\draw[very thick] (3)--(8);
\draw[very thick] (5)--(11);
\draw[very thick] (5)--(2);
\draw[very thick] (5)--(7);
\draw[very thick] (7)--(3);
\draw[very thick] (7)--(13);
\end{tikzpicture}
}
\caption{The $\tau$-tilting poset of the full DAG with framing given in Figure~\ref{fig:G(2,7) and Q}.
Note that only the non-exceptional routes are depicted, as every maximal clique contains the exceptional routes.
The exceptional routes for this example are given in Figure~\ref{fig:G(2,7) exceptionals}.}
\label{fig:G(2,7)poset}
\end{figure}

\begin{figure}
\begin{center}
\scalebox{0.7}{
\begin{tikzpicture}[scale=1]
\begin{scope}[scale=1.2, xshift=10, yshift=100]
	\vertex[fill,label=below:\footnotesize{$s$}](as) at (1,0) {};
	\vertex[fill,label=below:\footnotesize{$1$}](a1) at (2,0) {};
	\vertex[fill,label=below:\footnotesize{$2$}](a2) at (3,0) {};
	\vertex[fill,label=below:\footnotesize{$t$}](at) at (4,0) {};
	
	\draw[-stealth, thick] (as)--(a1);
	\draw[-stealth, thick] (a1)--(a2);
	\draw[-stealth, thick] (a2)--(at);

	\draw[-stealth, thick] (as) to[out=60,in=120] (a2);
	\draw[-stealth, thick] (a1) to[out=60,in=120] (at);
	\draw[-stealth, thick] (as) to[out=60,in=120] (a1);
	\draw[-stealth, thick] (a2) to[out=60,in=120] (at);

    \node[] at (1.5,-0.15) {\textcolor{cyan}{\tiny $1$}}; 
    \node[] at (2.5,-0.15) {\textcolor{cyan}{\tiny $1$}}; 
    \node[] at (3.5,-0.15) {\textcolor{cyan}{\tiny $1$}}; 
    
    \node[] at (2,0.75) {\textcolor{cyan}{\tiny $2$}}; 
    \node[] at (3,0.75) {\textcolor{cyan}{\tiny $2$}}; 
    
    \node[] at (1.5,0.17) {\textcolor{cyan}{\tiny $2$}}; 
    \node[] at (3.5,0.17) {\textcolor{cyan}{\tiny $2$}}; 
    
    \node[] at (1,1) {$G$};

\end{scope}

\begin{scope}[scale=1.2, xshift=38, yshift=40]
	\vertex[fill,label=below:\footnotesize{$1$}](a1) at (1,0) {};
	\vertex[fill,label=below:\footnotesize{$2$}](a2) at (2,0) {};
	
	\draw[-stealth, thick] (a1)--(a2);
	
    \node[] at (1.5,0.2) {$\alpha_1$};
    
    \node[] at (1,0.8) {$Q$};
\end{scope}

\begin{scope}[xshift=300, yshift=-150]
\begin{scope}[scale=0.8, xshift=0, yshift=0]
	\vertex[fill](as) at (1,0) {};
	\vertex[fill](a1) at (2,0) {};
	\vertex[fill](a2) at (3,0) {};
	\vertex[fill](at) at (4,0) {};
	
	\draw[-stealth, thick] (as)--(a1);

	\draw[-stealth, thick] (a1) to[out=60,in=120] (at);

	\node[] at (1.5,-0.2) {\textcolor{cyan}{\scriptsize $1$}}; 
	\node[] at (3,0.8) {\textcolor{cyan}{\scriptsize $2$}}; 
	\node[] at (5,0) {\large $P_1[1]$}; 
\end{scope}
\begin{scope}[scale=0.8, xshift=0, yshift=-30]
    \vertex[fill](as) at (1,0) {};
    \vertex[fill](a1) at (2,0) {};			\vertex[fill](a2) at (3,0) {};	
    \vertex[fill](at) at (4,0) {};
	
    \draw[-stealth, thick] (as)--(a1);
    \draw[-stealth, thick] (a1)--(a2);

    \draw[-stealth, thick] (a2) to[out=60,in=120] (at);
	
    \node[] at (1.5,-0.2) {\textcolor{cyan}{\scriptsize $1$}}; 
    \node[] at (2.5,-0.2) {\textcolor{cyan}{\scriptsize $1$}}; 
    
    \node[] at (3.5,0.1) {\textcolor{cyan}{\scriptsize $2$}}; 
    \node[] at (5,0) {\large $P_2[1]$}; 
\end{scope}
\end{scope}

\begin{scope}[xshift=200, yshift=0]
\begin{scope}[scale=0.8, xshift=0, yshift=0]
			\vertex[fill](as) at (1,0) {};
			\vertex[fill](a1) at (2,0) {};
			\vertex[fill](a2) at (3,0) {};
			\vertex[fill](at) at (4,0) {};
	
			\draw[-stealth, thick] (as)--(a1);

			\draw[-stealth, thick] (a1) to[out=60,in=120] (at);

			\node[] at (1.5,-0.2) {\textcolor{cyan}{\scriptsize $1$}}; 
    
   		\node[] at (3,0.8) {\textcolor{cyan}{\scriptsize $2$}}; 
    
    		\node[] at (5,0) {\large $P_1[1]$}; 
\end{scope}
\begin{scope}[scale=0.8, xshift=0, yshift=-30]
    \vertex[fill](as) at (1,0) {};
    \vertex[fill](as) at (1,0) {};
    \vertex[fill](a1) at (2,0) {};
    \vertex[fill](a2) at (3,0) {};
    \vertex[fill](at) at (4,0) {};
	
    \draw[-stealth, thick] (a2)--(at);
    \draw[-stealth, thick] (as) to[out=60,in=120] (a2);
    \node[] at (3.5,-0.2) {\textcolor{cyan}{\scriptsize $1$}}; 

    \node[] at (2,0.4) {\textcolor{cyan}{\scriptsize $2$}}; 
    \node[] at (5.2,0) {\large $M(\varepsilon_2)$}; 
\end{scope}
\end{scope}

\begin{scope}[xshift=500, yshift=-60]
\begin{scope}[scale=0.8, xshift=0, yshift=0]
			\vertex[fill](as) at (1,0) {};
			\vertex[fill](a1) at (2,0) {};
			\vertex[fill](a2) at (3,0) {};
			\vertex[fill](at) at (4,0) {};
	
			\draw[-stealth, thick] (a1)--(a2);
			\draw[-stealth, thick] (as) to[out=60,in=120] (a1);
			\draw[-stealth, thick] (a2) to[out=60,in=120] (at);
	
            \node[] at (2.5,-0.2) {\textcolor{cyan}{\scriptsize $1$}}; 
    	\node[] at (1.5,0.1) {\textcolor{cyan}{\scriptsize $2$}}; 
    	\node[] at (3.5,0.1) {\textcolor{cyan}{\scriptsize $2$}}; 
			
		\node[] at (5.2,0) {\large $M(\varepsilon_1)$}; 
		\end{scope}
		\begin{scope}[scale=0.8, xshift=0, yshift=-30]
		\vertex[fill](as) at (1,0) {};
			\vertex[fill](as) at (1,0) {};
			\vertex[fill](a1) at (2,0) {};
			\vertex[fill](a2) at (3,0) {};
			\vertex[fill](at) at (4,0) {};
	
			\draw[-stealth, thick] (as)--(a1);
			\draw[-stealth, thick] (a1)--(a2);

			\draw[-stealth, thick] (a2) to[out=60,in=120] (at);

			\node[] at (1.5,-0.2) {\textcolor{cyan}{\scriptsize $1$}}; 
    		\node[] at (2.5,-0.2) {\textcolor{cyan}{\scriptsize $1$}}; 
    
    		\node[] at (3.5,0.1) {\textcolor{cyan}{\scriptsize $2$}}; 
			
    		\node[] at (5,0) {\large $P_2[1]$}; 
\end{scope}
\end{scope}

\begin{scope}[xshift=500, yshift=60]
\begin{scope}[scale=0.8, xshift=0, yshift=0]
    \vertex[fill](as) at (1,0) {};
    \vertex[fill](a1) at (2,0) {};
    \vertex[fill](a2) at (3,0) {};
    \vertex[fill](at) at (4,0) {};
	
    \draw[-stealth, thick] (a1)--(a2);
    \draw[-stealth, thick] (as) to[out=60,in=120] (a1);
    \draw[-stealth, thick] (a2) to[out=60,in=120] (at);
	
    \node[] at (2.5,-0.2) {\textcolor{cyan}{\scriptsize $1$}}; 
    \node[] at (1.5,0.1) {\textcolor{cyan}{\scriptsize $2$}}; 
    \node[] at (3.5,0.1) {\textcolor{cyan}{\scriptsize $2$}}; 
			
    \node[] at (5.2,0) {\large $M(\varepsilon_1)$}; 
\end{scope}
\begin{scope}[scale=0.8, xshift=0, yshift=-30]
    \vertex[fill](as) at (1,0) {};
    \vertex[fill](as) at (1,0) {};
    \vertex[fill](a1) at (2,0) {};
    \vertex[fill](a2) at (3,0) {};
    \vertex[fill](at) at (4,0) {};
	
    \draw[-stealth, thick] (a1)--(a2);
    \draw[-stealth, thick] (a2)--(at);

    \draw[-stealth, thick] (as) to[out=60,in=120] (a1);
    
    \node[] at (2.5,-0.2) {\textcolor{cyan}{\scriptsize $1$}}; 
    \node[] at (3.5,-0.2) {\textcolor{cyan}{\scriptsize $1$}}; 
    
    \node[] at (1.5,0.1) {\textcolor{cyan}{\scriptsize $2$}}; 
    \node[] at (5.2,0) {\large $M(\alpha_1)$}; 
\end{scope}
\end{scope}

\begin{scope}[xshift=290, yshift=150]
\begin{scope}[scale=0.8, xshift=0, yshift=0]
    \vertex[fill](as) at (1,0) {};
    \vertex[fill](a1) at (2,0) {};
    \vertex[fill](a2) at (3,0) {};
    \vertex[fill](at) at (4,0) {};
	
    \draw[-stealth, thick] (a2)--(at);
    \draw[-stealth, thick] (as) to[out=60,in=120] (a2);
    
    \node[] at (3.5,-0.2) {\textcolor{cyan}{\scriptsize $1$}}; 
    \node[] at (2,0.8) {\textcolor{cyan}{\scriptsize $2$}}; 
    
    \node[] at (5.2,0) {\large $M(\varepsilon_2)$}; 
\end{scope}
\begin{scope}[scale=0.8, xshift=0, yshift=-30]
    \vertex[fill](as) at (1,0) {};
    \vertex[fill](as) at (1,0) {};
    \vertex[fill](a1) at (2,0) {};
    \vertex[fill](a2) at (3,0) {};
    \vertex[fill](at) at (4,0) {};
	
    \draw[-stealth, thick] (a1)--(a2);
    \draw[-stealth, thick] (a2)--(at);

    \draw[-stealth, thick] (as) to[out=60,in=120] (a1);
    \node[] at (2.5,-0.2) {\textcolor{cyan}{\scriptsize $1$}}; 
    \node[] at (3.5,-0.2) {\textcolor{cyan}{\scriptsize $1$}}; 
    
    \node[] at (1.5,0.1) {\textcolor{cyan}{\scriptsize $2$}}; 
    
    \node[] at (5.2,0) {\large $M(\alpha_1)$}; 
\end{scope}
\end{scope}

\begin{scope}[xshift=440, yshift=-150]
    \draw[thick] (0,0) -- (3,1.5); 
\end{scope}
\begin{scope}[xshift=330, yshift=-130]
    \draw[thick] (0,0) -- (-1.3,3); 
\end{scope}
\begin{scope}[xshift=290, yshift=25]
    \draw[thick] (0,0) -- (1.3,3); 
\end{scope}
\begin{scope}[xshift=440, yshift=125]
    \draw[thick] (0,0) -- (3,-1.5); 
\end{scope}
\begin{scope}[xshift=550, yshift=-40]
    \draw[thick] (0,0) -- (0,2); 
\end{scope}

\begin{scope}[xshift=440, yshift=-160]
    \node[] (a) at (2.5,0.6) {$(P_1[1],M(\varepsilon_1))$}; 
    \vertex[fill,label=below:\footnotesize{$1$}](b2) at (2,0) {};
\end{scope}
\begin{scope}[xshift=190, yshift=-100]
    \node[] (a) at (2.5,0.6) {$(P_2[1],M(\varepsilon_2))$}; 
    \vertex[fill,label=below:\footnotesize{$2$}](b1) at (3,0) {};    
\end{scope}
\begin{scope}[xshift=520, yshift=-15]
    \node[] (a) at (2.5,0.6) {$(P_2[1],M(\alpha_1))$}; 
    \vertex[fill,label=below:\footnotesize{$1$}](b1) at (2,0) {};
    \vertex[fill,label=below:\footnotesize{$2$}](b2) at (3,0) {};    
    \draw[-stealth, thick] (b1)--(b2);
\end{scope}
\begin{scope}[xshift=190, yshift=70]
    \node[] (a) at (2.5,0.6) {$(P_1[1], M(\alpha_1))$}; 
    \vertex[fill,label=below:\footnotesize{$1$}](b1) at (2,0) {};    
\end{scope}
\begin{scope}[xshift=440, yshift=110]
    \node[] (a) at (2.5,0.6) {$(M(\varepsilon_1),M(\varepsilon_2))$}; 
    \vertex[fill,label=below:\footnotesize{$2$}](b2) at (3,0) {};
\end{scope}
\end{tikzpicture}
}
\end{center}
\caption{A graph $G$ obtained as a full contraction of $G(2,6)$ with the length framing and its induced quiver $Q$ (left).
The associated $\tau$-tilting poset with vertices given by maximal cliques (right).
The exceptional routes are omitted from each clique and the poset edges are labeled with the associated brick and connected component $w$.
}
\label{fig:G(2,6) and poset}
\end{figure}

\begin{remark}
We say that the framed DAG $[G,F]$ is {\em symmetric} if reversing its vertex labeling ($i\mapsto n-i$) in $[G,F]$ is a framing-preserving isomorphism of $[G,F]$. 
In a symmetric full DAG $G$ with framing $F$, if $\Delta_1 \prec \Delta_2$ in $\DKK(G,F)$, then reversing the vertex labels in the routes of $\Delta_1$ and $\Delta_2$ yield two cliques $\Delta_1'$ and $\Delta_2'$ in $\DKK(G,F)$ satisfying $\Delta_2' \prec \Delta_1'$. 
Thus we observe that the poset in Definition~\ref{def:order} is self-dual if $[G,F]$ is symmetric.
We see this in Figure~\ref{fig:G(2,7)poset} as $G(2,7)$ is symmetric with the ample framing of Figure~\ref{fig:G(2,7) and Q}. 
\end{remark}

Following Definition~\ref{def:order} we can label every edge of the dual graph of  $\DKK(G,F)$ connecting two adjacent cliques $\Delta_1$ and $\Delta_2$ by the unique path $w$.
The corresponding module $M(w)$ in the blossoming algebra is a \emph{brick}, meaning that the only morphisms from $M(w)$ to itself are isomorphisms and the zero map.
Indeed, this follows because $M(w)$ is at most one-dimensional at every vertex of the quiver $\widehat{Q}$.
On the level of the representation theory, this corresponds to the so-called brick labeling of the edges of the $\tau$-tilting poset $\text{st}(\Lambda(G))$, which was studied for general finite dimensional algebras in \cite{BCZ19,DIRRT}.  In particular, the following statement is a special case of \cite[Proposition 3.2.5]{BTZ} about torsion classes whenever the torsion class is generated by a support $\tau$-tilting module.  

\begin{proposition}\cite[Proposition 3.2.5]{BTZ}
Let $\Lambda$ be a finite dimensional algebra.  Then a support $\tau$-tilting module $T \in \text{st}(\Lambda)$ is completely determined by the bricks labeling the down edges coming out of $T$ in the support $\tau$-tilting poset.
\end{proposition}


Similarly, $T$ is completely determined by the bricks labeling the up edges coming into $T$.
This leads to the definition of the kappa map on the $\tau$-tilting poset introduced and studied in \cite{BTZ}.

\begin{definition}\cite[Proposition B]{BTZ}
Let $\Lambda$ be a finite dimensional algebra.
The \emph{kappa map} $\kappa: \text{st}(\Lambda)\to\text{st}(\Lambda)$ on the support $\tau$-tilting poset is defined as follows.
Given $T\in \text{st}(\Lambda)$, let $M_1, \dots, M_t$ be the set of bricks labeling the down edges coming out of $T$.
Then $\kappa(T)$ is defined as the support $\tau$-tilting module with up edges having labels $M_1, \dots, M_t$.  
\end{definition}

The following statement follows directly from the definition of the map given above and the fact that every module in $\text{st}(\Lambda)$ is uniquely determined by the bricks labeling its up edges or its down edges.     

\begin{theorem}\label{thm:kappa}
Let $\Lambda$ be an algebra such that its $\tau$-titling poset is finite.  Then the kappa map is a bijection on $\textup{st}(\Lambda)$.  
\end{theorem}

Given an element $p$ of a poset $P$, we define the following statistics.
Let $\textup{dcov}(p)$ denote the number of down neighbors of $p$ in $P$, and similarly let $\textup{ucov}(p)$ denote the number of up neighbors of $p$ in $P$.
If $P$ be a finite poset, then we can define the following polynomial $\textup{dcov}(P) = \sum_{p\in P} x^{\textup{dcov}(p)}$ which records the down statistics of $P$.  

\begin{theorem}\label{thm:downsymmetry}
Let $\Lambda$ be an algebra such that its $\tau$-titling poset is finite, then the polynomial $\textup{dcov}(\textup{st}(\Lambda))$ is symmetric. 
\end{theorem}

\begin{proof}
Let $T$ be a support $\tau$-tilting $\Lambda$-module. Then, $\textup{dcov}(T)=\textup{ucov}(\kappa(T))$ by the definition of the $\kappa$-map. 
Moreover, by Thoerem~\ref{thm:kappa} this map is a bijection on $\text{st}(\Lambda)$, so we obtain a bijection on the following subsets of support $\tau$-tilting modules for all $r$:
\[ \{T \in \text{st}(\Lambda) :  \textup{dcov}(T)=r\}  \longleftrightarrow \{T \in \text{st}(\Lambda) :  \textup{ucov}(T)=r\} .\]
By \cite[Theorem 2.18]{AIR}, the Hasse diagram of the $\tau$-titling poset is $n$-regular, where $n$ is the number of vertices of the quiver of $\Lambda$. 
This means that every vertex of $\text{st}(\Lambda)$ has exactly $n$ neighbors.
This implies that the following sets are equal:
\[   \{T \in \text{st}(\Lambda) :  \textup{ucov}(T)=r\}= \{T \in \text{st}(\Lambda) :  \textup{dcov}(T)=n-r\}.\] 
Therefore, there is a bijection between the modules in $\text{st}(\Lambda)$ that have $r$ down neighbors and the modules that have $n-r$ down neighbors.
This shows that the polynomial $\textup{dcov}(\text{st}(\Lambda))$ is symmetric.  
\end{proof}

\begin{remark} 
Although the kappa map is a bijection on $\textup{st}(\Lambda)$, in general it is not a poset morphism.
Consider the $\tau$-tilting poset of $G=G(2,7)$ in Figure~\ref{fig:G(2,7)poset}, and the four elements of the poset depicted in Figure~\ref{fig:notaposetmap}.
The maximal clique $\Delta_1$ covers $\Delta_2$ in the $\tau$-tilting poset, but $\Delta_3=\kappa(\Delta_1)$ does not cover $\Delta_4 = \kappa(\Delta_2)$.
\end{remark}

\begin{figure}
\centering
\scalebox{0.65}{
\begin{tikzpicture}
    \node[] at (4, 14) {$\Delta_1$};
    \node[] (3) at (6,13.5) {
    \begin{tikzpicture}
    \begin{scope}[scale=0.5, xshift=0, yshift=0]
    	\vertex[fill](as) at (1,0) {};
    	\vertex[fill](a1) at (2,0) {};
    	\vertex[fill](a2) at (3,0) {};
    	\vertex[fill](a3) at (4,0) {};
    	\vertex[fill](at) at (5,0) {};
    	\draw[thick] (a3)--(at);
        \node[] at (4.5,0.3) {\textcolor{cyan}{\tiny$1$}}; 
    	\draw[thick] (as) to[out=60,in=120] (a1);
    	\node[] at (1.5,0.7) {\textcolor{cyan}{\tiny$2$}}; 
    	\draw[thick] (a1) to[out=60,in=120] (a3);
    	\node[] at (3, 0.9)  {\textcolor{cyan}{\tiny$2$}}; 
    \end{scope}
    \begin{scope}[scale=0.5, xshift=0, yshift=40]
    	\vertex[fill](as) at (1,0) {};
    	\vertex[fill](a1) at (2,0) {};
    	\vertex[fill](a2) at (3,0) {};
    	\vertex[fill](a3) at (4,0) {};
    	\vertex[fill](at) at (5,0) {};
    	\draw[thick] (a1)--(a2);
    	\node[] at (2.5,0.3) {\textcolor{cyan}{\tiny$1$}};  
    	\draw[thick] (a2)--(a3);
    	\node[] at (3.5,0.3) {\textcolor{cyan}{\tiny$1$}};  
    	\draw[thick] (a3)--(at);
        \node[] at (4.5,0.3) {\textcolor{cyan}{\tiny$1$}}; 
    	\draw[thick] (as) to[out=60,in=120] (a1);
    	\node[] at (1.5,0.7) {\textcolor{cyan}{\tiny$2$}}; 
    \end{scope}
    \begin{scope}[scale=0.5, xshift=0, yshift=80]
    	\vertex[fill](as) at (1,0) {};
    	\vertex[fill](a1) at (2,0) {};
    	\vertex[fill](a2) at (3,0) {};
    	\vertex[fill](a3) at (4,0) {};
    	\vertex[fill](at) at (5,0) {};
    	\draw[thick] (a1)--(a2);
    	\node[] at (2.5,0.3) {\textcolor{cyan}{\tiny$1$}};  
    	\draw[thick] (a2) to[out=60,in=120] (at);
    	\node[] at (4,0.9) {\textcolor{cyan}{\tiny$2$}}; 
    	\draw[thick] (as) to[out=60,in=120] (a1);
    	\node[] at (1.5,0.7) {\textcolor{cyan}{\tiny$2$}}; 
    \end{scope}
    \end{tikzpicture}
    };  
    
    \node[] at (13, 10.5) {$\Delta_2$};
    \node[] (7) at (11,10.5) {
    \begin{tikzpicture}
    \begin{scope}[scale=0.5, xshift=0, yshift=40]
	    \vertex[fill](as) at (1,0) {};
	    \vertex[fill](a1) at (2,0) {};
    	\vertex[fill](a2) at (3,0) {};
    	\vertex[fill](a3) at (4,0) {};
    	\vertex[fill](at) at (5,0) {};
    	\draw[thick] (a1)--(a2);
    	\node[] at (2.5,0.3) {\textcolor{cyan}{\tiny$1$}};  
    	\draw[thick] (a2)--(a3);
    	\node[] at (3.5,0.3) {\textcolor{cyan}{\tiny$1$}};  
    	\draw[thick] (as) to[out=60,in=120] (a1);
    	\node[] at (1.5,0.7) {\textcolor{cyan}{\tiny$2$}}; 
    	\draw[thick] (a3) to[out=60,in=120] (at);
    	\node[] at (4.5,0.7) {\textcolor{cyan}{\tiny $2$}}; 
    \end{scope}
    \begin{scope}[scale=0.5, xshift=0, yshift=0]
    	\vertex[fill](as) at (1,0) {};
    	\vertex[fill](a1) at (2,0) {};
    	\vertex[fill](a2) at (3,0) {};
    	\vertex[fill](a3) at (4,0) {};
    	\vertex[fill](at) at (5,0) {};
    	\draw[thick] (a1)--(a2);
    	\node[] at (2.5,0.3) {\textcolor{cyan}{\tiny$1$}};  
    	\draw[thick] (a2)--(a3);
    	\node[] at (3.5,0.3) {\textcolor{cyan}{\tiny$1$}};  
    	\draw[thick] (a3)--(at);
        \node[] at (4.5,0.3) {\textcolor{cyan}{\tiny$1$}}; 
    	\draw[thick] (as) to[out=60,in=120] (a1);
    	\node[] at (1.5,0.7) {\textcolor{cyan}{\tiny$2$}}; 
    \end{scope}
    \begin{scope}[scale=0.5, xshift=0, yshift=80]
    	\vertex[fill](as) at (1,0) {};
    	\vertex[fill](a1) at (2,0) {};
    	\vertex[fill](a2) at (3,0) {};
    	\vertex[fill](a3) at (4,0) {};
    	\vertex[fill](at) at (5,0) {};
    	\draw[thick] (a1)--(a2);
    	\node[] at (2.5,0.3) {\textcolor{cyan}{\tiny$1$}};  
    	\draw[thick] (a2) to[out=60,in=120] (at);
    	\node[] at (4,0.9) {\textcolor{cyan}{\tiny$2$}}; 
    	\draw[thick] (as) to[out=60,in=120] (a1);
    	\node[] at (1.5,0.7) {\textcolor{cyan}{\tiny$2$}}; 
    \end{scope}
    \end{tikzpicture}
    };            

   \node[] at (-1, 10.5) {$\Delta_3$};
    \node[] (14) at (1,10.5) {
    \begin{tikzpicture}
    \begin{scope}[scale=0.5, xshift=0, yshift=0]
    	\vertex[fill](as) at (1,0) {};
    	\vertex[fill](a1) at (2,0) {};
    	\vertex[fill](a2) at (3,0) {};
    	\vertex[fill](a3) at (4,0) {};
    	\vertex[fill](at) at (5,0) {};
    	\draw[thick] (as)--(a1);
    	\node[] at (1.5,0.3) {\textcolor{cyan}{\tiny$1$}}; 
    	\draw[thick] (a1)--(a2);
    	\node[] at (2.5,0.3) {\textcolor{cyan}{\tiny$1$}};  
    	\draw[thick] (a2)--(a3);
    	\node[] at (3.5,0.3) {\textcolor{cyan}{\tiny$1$}};  
    	\draw[thick] (a3) to[out=60,in=120] (at);
    	\node[] at (4.5,0.7) {\textcolor{cyan}{\tiny $2$}}; 
    \end{scope}
    \begin{scope}[scale=0.5, xshift=0, yshift=40]
    	\vertex[fill](as) at (1,0) {};
    	\vertex[fill](a1) at (2,0) {};
    	\vertex[fill](a2) at (3,0) {};
    	\vertex[fill](a3) at (4,0) {};
    	\vertex[fill](at) at (5,0) {};
    	\draw[thick] (as)--(a1);
    	\node[] at (1.5,0.3) {\textcolor{cyan}{\tiny$1$}}; 
    	\draw[thick] (a1)--(a2);
    	\node[] at (2.5,0.3) {\textcolor{cyan}{\tiny$1$}};  
    	\draw[thick] (a2) to[out=60,in=120] (at);
    	\node[] at (4,0.9) {\textcolor{cyan}{\tiny$2$}}; 
    \end{scope}
    \begin{scope}[scale=0.5, xshift=0, yshift=80]
    	\vertex[fill](as) at (1,0) {};
    	\vertex[fill](a1) at (2,0) {};
    	\vertex[fill](a2) at (3,0) {};
    	\vertex[fill](a3) at (4,0) {};
    	\vertex[fill](at) at (5,0) {};
    	\draw[thick] (a1)--(a2);
    	\node[] at (2.5,0.3) {\textcolor{cyan}{\tiny$1$}};  
    	\draw[thick] (a2) to[out=60,in=120] (at);
    	\node[] at (4,0.9) {\textcolor{cyan}{\tiny$2$}}; 
    	\draw[thick] (as) to[out=60,in=120] (a1);
    	\node[] at (1.5,0.7) {\textcolor{cyan}{\tiny$2$}}; 
    \end{scope}
    \end{tikzpicture}
    };        
    
    \node[] at (8, 6.5) {$\Delta_4$};
    \node[] (11) at (6,7.5) {
    \begin{tikzpicture}
    \begin{scope}[scale=0.5, xshift=0, yshift=0]
	\vertex[fill](as) at (1,0) {};
	\vertex[fill](a1) at (2,0) {};
	\vertex[fill](a2) at (3,0) {};
	\vertex[fill](a3) at (4,0) {};
	\vertex[fill](at) at (5,0) {};
	\draw[thick] (as)--(a1);
	\node[] at (1.5,0.3) {\textcolor{cyan}{\tiny$1$}}; 
	\draw[thick] (a1)--(a2);
	\node[] at (2.5,0.3) {\textcolor{cyan}{\tiny$1$}};  
	\draw[thick] (a2)--(a3);
	\node[] at (3.5,0.3) {\textcolor{cyan}{\tiny$1$}};  
	\draw[thick] (a3) to[out=60,in=120] (at);
	\node[] at (4.5,0.7) {\textcolor{cyan}{\tiny $2$}}; 
    \end{scope}
    \begin{scope}[scale=0.5, xshift=0, yshift=40]
	\vertex[fill](as) at (1,0) {};
	\vertex[fill](a1) at (2,0) {};
	\vertex[fill](a2) at (3,0) {};
	\vertex[fill](a3) at (4,0) {};
	\vertex[fill](at) at (5,0) {};
	\draw[thick] (a2)--(a3);
	\node[] at (3.5,0.3) {\textcolor{cyan}{\tiny$1$}};  
    \draw[thick] (as) to[out=60,in=120] (a2);
    \node[] at (2,0.9) {\textcolor{cyan}{\tiny$2$}}; 
	\draw[thick] (a3) to[out=60,in=120] (at);
	\node[] at (4.5,0.7) {\textcolor{cyan}{\tiny $2$}}; 
    \end{scope}
    \begin{scope}[scale=0.5, xshift=0, yshift=80]
	\vertex[fill](as) at (1,0) {};
	\vertex[fill](a1) at (2,0) {};
	\vertex[fill](a2) at (3,0) {};
	\vertex[fill](a3) at (4,0) {};
	\vertex[fill](at) at (5,0) {};
	\draw[thick] (a1)--(a2);
	\node[] at (2.5,0.3) {\textcolor{cyan}{\tiny$1$}};  
	\draw[thick] (a2)--(a3);
	\node[] at (3.5,0.3) {\textcolor{cyan}{\tiny$1$}};  
	\draw[thick] (as) to[out=60,in=120] (a1);
	\node[] at (1.5,0.7) {\textcolor{cyan}{\tiny$2$}}; 
	\draw[thick] (a3) to[out=60,in=120] (at);
	\node[] at (4.5,0.7) {\textcolor{cyan}{\tiny $2$}}; 
    \end{scope}
    \end{tikzpicture}
    };

\draw[thick] (7)--(3);
\draw[-stealth, very thick] (3) to[out=-90,in=0] (14);
    \node[] at (4,11) {$\kappa$};
\draw[-stealth, very thick] (7) to[out=-90,in=0] (11);
    \node[] at (9,8) {$\kappa$};
\end{tikzpicture}
}
\caption{An example showing that $\kappa$ is not an order-preserving map on the $\tau$-tilting poset.}
\label{fig:notaposetmap}
\end{figure}
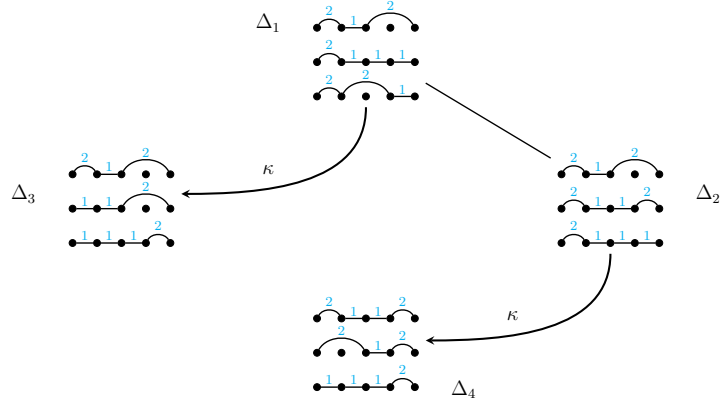


Given a lattice polytope $P$ of dimension $d$ in $\R^n$, the \emph{Ehrhart series} of $P$ is the rational generating function
\[
\mathrm{Ehr}_P(x)=\sum_{t\geq 0}|tP\cap \Z^{n}|x^t=\frac{\sum_{i=0}^dh_i^*x^i}{(1-x)^{d+1}}\, .
\]
The rationality of $\mathrm{Ehr}_P(x)$ is due to Ehrhart~\cite{Ehrhart}, and it is known by work of Stanley~\cite{StanleyDecompositions} that the vector of coefficients $h_P^*:=(h_0^*,\ldots,h_d^*)$, called the \emph{$h^*$-vector of $P$}, consists of nonnegative integers with $h^*_0 = 1$.
By defining the \emph{cone over $P$} to be 
\[
\cone(P):=\mathrm{span}_{\R_{\geq 0}}\{(1,\mathbf{p}):\mathbf{p}\in P\} \, ,
\]
one can show that the Ehrhart series for $P$ is the Hilbert series for the semigroup algebra of $\cone(P)$ with grading given by the first coordinate.

When $P$ admits a unimodular triangulation $T$, the $h^*$-vector of $P$ and the $h$-vector of $T$ coincide \cite[Theorem 10.3]{BR07}.
The $h$-vector of a shellable triangulation has nonnegative entries which can be computed combinatorially from the shelling order as follows. 
For a fixed shelling order $F_1,...,F_s$ on the facets of the triangulation, the restriction $R_j$ of the facet $F_j$ is defined to be the set 
\[
R_j := \{v\in F_j : v \text{ is a vertex in }F_j \text{ and } F_j \setminus v \subseteq F_i \text{ for some } 1 \leq i < j\} \, .
\]
The $i$-th entry of the $h$-vector is then given by $h_i = |\{j:|R_j| = i, 1 \leq j \leq s\}|$. 
Thus, if the dual graph of the triangulation admits the structure of a Hasse diagram of a poset with at least one linear extension giving a shelling order of the triangulation, then $h_i$ is the number of elements in the poset covering exactly $i$ elements.
Since the $\tau$-tilting complex is shellable by Theorem~\ref{thm:tau-shellable}, we obtain the following lemma. 

\begin{lemma}
\label{lem:h*coverposetcount}
Let $Q$ be the $\tau$-tilting poset associated with $\DKK(G,F)$.
Then the $i$-th coefficient of the $h^*$-vector of $\calF_1(G)$ is given by the number elements in $Q$ covering exactly $i$ elements.
\end{lemma}

\begin{corollary}
\label{cor:h*symmetry}
If $G$ is a full DAG, then $\calF_1(G)$ has a symmetric $h^*$-vector.
\end{corollary}

\begin{proof}
This follows from Lemma~\ref{lem:h*coverposetcount} and Theorem~\ref{thm:downsymmetry}.
\end{proof}

\begin{example}
Consider the full contraction of $G(2,7)$ with the framing specified in Figure~\ref{fig:G(2,7) and Q}, having $\tau$-tilting poset as given in Figure~\ref{fig:G(2,7)poset}.
By inspection, for this poset we have 
\[
\textup{dcov}(P) = \sum_{p\,\in \, P} x^{\textup{dcov}(p)}=1+7x+7x^2+x^3 \, ,
\]
since for example there are seven elements of the poset covering a single element, one element covering three elements, etc.
Thus, the $h^*$-polynomial for the corresponding flow polytope is symmetric and unimodal.
\end{example}

The symmetry of the $h^*$-polynomial of a lattice polytope has known geometric consequences, which we discuss next.

\begin{definition}\label{def:gorensteinpolytope}
A lattice polytope $P$ is \emph{reflexive} if there exists an integer vector $\mathbf{v}$ and an integer matrix $A$ such that $P+\mathbf{v}=\{\bx \in \R^n: A\bx \leq \mathbf{1}\}$, where $\mathbf{1}$ denotes the vector with all entries equal to $1$.
A lattice polytope $P\subset \R^n$ is \emph{Gorenstein of index $k$} if $kP$ is a reflexive polytope.
\end{definition}

In greater generality, a pointed rational cone $C$ is \emph{Gorenstein} if there exists an integer point $\mathbf{c}$ satisfying
\[
\mathbf{c}+\left( \Z^{1+n}\cap C \right) = \Z^{1+n}\cap C^\circ,
\]
where $C^\circ$ denotes the interior of the cone $C$.
It is known that $P$ is Gorenstein if and only if $\cone(P)$ is Gorenstein.
An alternative characterization of the Gorenstein condition is the following.

\begin{lemma}[Bruns, R\"omer~\cite{BrunsRomer}]\label{lemma:gor}
Let $C$ be a pointed rational cone with supporting hyperplanes of the form $\sigma\cdot x\geq 0$ where $\sigma$ is a vector of integers such that the greatest common divisor of the entries in $\sigma$ is $1$.
For such a cone, $C$ is Gorenstein if and only if there exists an integer point $\bc$ in the interior of $C$ such that $\sigma\cdot \bc=1$ for all supporting hyperplanes $\sigma$ of $C$.
\end{lemma}

This lemma leads to a characterization of Gorenstein flow polytopes.

\begin{proposition}\label{prop:gor}
The flow polytope $\calF_1(G)$ is Gorenstein if and only if $G$ is a DAG such that for each inner vertex $v$ of $G$ the in-degree and out-degree of $v$ are equal.
\end{proposition}

\begin{proof}
Note that the cone over $\calF_1(G)$ is equivalent to $\calF(G)$, hence we work in this setting.
Since the supporting hyperplanes of $\calF(G)$ are all of the form $x_e\geq 0$ for the edges $e$ of $G$, by Lemma~\ref{lemma:gor} the only candidate for a Gorenstein point $\bc$ is the all-ones vector. 
The all-ones vector is in $\calF(G)$ if and only if the equality of in- and out-degree holds for each inner vertex $v$ of $G$.
\end{proof}

Yet another classification of Gorenstein polytopes is provided by symmetry of coefficients of $h^*$-polynomials, as follows.
\begin{theorem}[Stanley~\cite{stanleyhilbertgraded}]\label{thm:gorensteinsymmetry}
A $d$-dimensional lattice polytope $P$ with 
\[
h^*_P=(h_1^*,\ldots,h_s^*,0,\ldots,0)\in \Z^d \, ,
\]
where $h^*_s\neq 0$, is Gorenstein if and only if $h^*_i=h^*_{s-i}$ for all $i$.
\end{theorem}

We therefore have two proofs of the following theorem.

\begin{theorem}\label{thm:fullgorenstein}
If $G$ is a full DAG, then $\calF_1(G)$ is Gorenstein.
\end{theorem}

\begin{proof}
For the first proof, apply Corollary~\ref{cor:h*symmetry} and Theorem~\ref{thm:gorensteinsymmetry} regarding the symmetry of the $h^*$-vector.
For the second proof, since the in- and out-degrees of every interior vertex are equal to $2$, Proposition~\ref{prop:gor} is satisifed.
\end{proof}

\begin{remark}\label{rem:Gkn+k}
In~\cite[Exercise 4.56(c)]{stanleyEC1}, Stanley introduces a class of polytopes commonly referred to as consecutive coordinate polytopes.
Theorem~\ref{thm:fullgorenstein} generalizes a result of Ayyer, Josuat-Verg\`{e}s, and Ramassamy~\cite[Theorem 2.10]{AJR} which states the consecutive coordinate polytope, denoted $\widehat{B}_{k,n}$, has a palindromic $h^*$-vector.
They give a formula for computing the $h^*$-polynomial of $\widehat{B}_{k,n}$ as the generating function of total cyclic orders with resepct to the number of descents.
It was shown in~\cite{GHMY} that $\widehat{B}_{k,n}$ is integrally equivalent to the flow polytope for the DAG $G(k,n+k)$ studied in Section~\ref{sec:enumerating}, and furthermore, a consequence of \cite[Theorem 4.8]{GHMY} and Theorem~\ref{maximal_cliques} gives a bijection between total cyclic orders $A_{k, n+k}$ considered by Ayyer et al., and support $\tau$-tilting modules $\textup{st}(\Lambda(G(k,n+k)))$.
Interestingly, this bijection is not weight-preserving between descents of total cyclic orders and $\textup{dcov}$ of elements in the $\tau$-tilting poset.
Ayyer et al. ask if the palindromicity result of the $h^*$-vector of $\widehat{B}_{k,n}$ can be understood on a combinatorial level via an involution on total cyclic orders.  
The kappa map on $\textup{st}(\Lambda)$ provides such an answer in terms of support $\tau$-tilting modules.
\end{remark}

The Gorenstein condition on $\calF_1(G)$ combined with the fact that the DKK triangulation is regular and unimodular allows us to apply the following theorem from~\cite[Theorem 1]{BrunsRomer}.

\begin{theorem}[Bruns, R\"{o}mer~\cite{BrunsRomer}] \label{thm:brunsroemer}
Let $P$ be a lattice polytope such that $P$ admits a regular unimodular triangulation and $P$ is Gorenstein.
Then the $h^*$-vector for $P$ is unimodal.
\end{theorem}

\begin{corollary}\label{cor:fullhstarunimodal}
If $G$ is a full DAG, then $\calF_1(G)$ is $h^*$-unimodal.
\end{corollary}
\begin{proof}
Recall that $\calF_1(G)$ admits a regular unimodular triangulation via Theorem~\ref{thm:DKKtriangulation}. Thus, the result follows from Theorem~\ref{thm:gorensteinsymmetry} and Theorem~\ref{thm:brunsroemer}.
\end{proof}

It would be of interest to investigate other properties such as log-concavity, real-rootedness, and $\gamma$-non-negativity for these polytopes.

From our enumeration of ample framings, we can show that the Gorenstein polytope $\calF_1(G)$ for a full DAG $G$ has many special simplices, defined as follows.
The concept of a special simplex was originated by Athanasiadis~\cite{Athanasiadis}.

\begin{definition}
\label{def:specialsimplex}
Given a Gorenstein polytope $P$ and a simplex $S$ with vertices lattice points in $P$, we say $S$ is \emph{special} if the intersection of any facet of $P$ with $S$ is a facet of $S$.
\end{definition}

It is known that given a Gorenstein polytope $P$ having the integer decomposition property and a special simplex $S$, projecting $P$ along the affine span of $S$ yields a reflexive polytope with the same $h^*$-vector as $P$; this is the key ingredient of the proof of Theorem~\ref{thm:brunsroemer} by Bruns and R\"{o}mer~\cite[Corollary 4]{BrunsRomer}.
It is known that every lattice polytope with a unimodular triangulation has the integer decomposition property.
Thus, understanding the special simplices in polytopes with regular triangulations is of interest.

\begin{theorem}
\label{thm:fullspecial}
Given a full DAG $G$ with an ample framing $F$, the set of exceptional routes forms a special simplex for $\calF_1(G)$.
\end{theorem}

\begin{proof}
Let $P=\calF_1(G)$.
By Proposition~\ref{prop:contractidle}, every facet of $P$ is of the form $x_e= 0$.
Since $F$ is an ample framing, every edge is contained in a unique exceptional route $R$.
Thus, for each facet $H$ of $P$ with $x_e=0$, the vertices of $P$ contained in $H$ are exactly those routes that do not contain $e$.
Since there is exactly one exceptional route containing $e$, say $R_e$, all exceptional routes except $R_e$ are contained in $H$.
\end{proof}

\begin{example}
In Figure~\ref{fig:G(2,7) exceptionals}, we see the three exceptional routes for the framed DAG from Figure~\ref{fig:G(2,7) and Q}.
Note that the facets for the flow polytope of $G(2,7)$ correspond to $x_e=0$ for each edge $e$.
For each fixed edge $e$ in $G(2,7)$, two of the routes in Figure~\ref{fig:G(2,7) exceptionals} do not contain that edge, and hence the line segment between those two routes in the flow polytope are contained in the corresponding facet.
Thus, every facet of the flow polytope for $G(2,7)$ intersects the triangle formed by $R_1$, $R_2$, and $R_3$ in an edge, making this triangle a special simplex.
\end{example}
 
 \begin{figure}
\begin{center}
    \begin{tikzpicture}
\begin{scope}[scale=1, xshift=-150, yshift=0]
	\vertex[fill,label=below:\footnotesize{$s$}](as) at (1,0) {};
 	\vertex[fill,label=below:\footnotesize{$1$}](a1) at (2,0) {};
	\vertex[fill,label=below:\footnotesize{$2$}](a2) at (3,0) {};
 	\vertex[fill,label=below:\footnotesize{$3$}](a3) at (4,0) {};
	\vertex[fill,label=below:\footnotesize{$t$}](at) at (5,0) {};
	

	\draw[-stealth, thick] (as) to[out=60,in=120] (a2);
	\draw[-stealth, thick] (a2) to[out=60,in=120] (at);

	
    
    \node[] at (2,0.75) {\textcolor{cyan}{\tiny$2$}}; 
    \node[] at (4,0.75) {\textcolor{cyan}{\tiny$2$}}; 
    
    
    \node[] at (1,1) {$R_1$};

\end{scope}

\begin{scope}[scale=1, xshift=0, yshift=0]
	\vertex[fill,label=below:\footnotesize{$s$}](as) at (1,0) {};
	\vertex[fill,label=below:\footnotesize{$1$}](a1) at (2,0) {};
	\vertex[fill,label=below:\footnotesize{$2$}](a2) at (3,0) {};
	\vertex[fill,label=below:\footnotesize{$3$}](a3) at (4,0) {};
	\vertex[fill,label=below:\footnotesize{$t$}](at) at (5,0) {};
	
	\draw[-stealth, thick] (as)--(a1);
	\draw[-stealth, thick] (a1)--(a2);
	\draw[-stealth, thick] (a2)--(a3);
	\draw[-stealth, thick] (a3)--(at);


	
    \node[] at (1.5,0.15) {\textcolor{cyan}{\tiny$1$}}; 
    \node[] at (2.5,0.15) {\textcolor{cyan}{\tiny$1$}}; 
    \node[] at (3.5,0.15) {\textcolor{cyan}{\tiny$1$}}; 
    \node[] at (4.5,0.15) {\textcolor{cyan}{\tiny$1$}}; 
    
    
    
    \node[] at (1,1) {$R_2$};

\end{scope}

\begin{scope}[scale=1, xshift=150, yshift=0]
	\vertex[fill,label=below:\footnotesize{$s$}](as) at (1,0) {};
	\vertex[fill,label=below:\footnotesize{$1$}](a1) at (2,0) {};
 	\vertex[fill,label=below:\footnotesize{$2$}](a2) at (3,0) {};
	\vertex[fill,label=below:\footnotesize{$3$}](a3) at (4,0) {};
	\vertex[fill,label=below:\footnotesize{$t$}](at) at (5,0) {};
	


	\draw[-stealth, thick] (as) to[out=60,in=120] (a1);
	\draw[-stealth, thick] (a1) to[out=60,in=120] (a3);
	\draw[-stealth, thick] (a3) to[out=60,in=120] (at);
	
    
    
    \node[] at (1.5,0.55) {\textcolor{cyan}{\tiny$2$}}; 
    \node[] at (3, 0.75)  {\textcolor{cyan}{\tiny$2$}}; 
    \node[] at (4.5,0.55) {\textcolor{cyan}{\tiny$2$}}; 
    
    \node[] at (1,1) {$R_3$};
    
\end{scope}

\end{tikzpicture}
\end{center}
    \caption{\(R_1,\, R_2,\, R_3\) are the exceptional routes of the framed full DAG given in Figure~\ref{fig:G(2,7) and Q}.}
    \label{fig:G(2,7) exceptionals}
\end{figure}
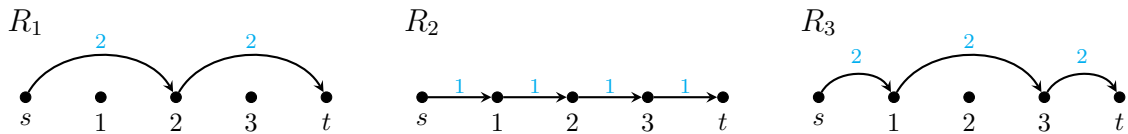

\begin{corollary}
\label{cor:specialcount}
Given a full DAG $G$, with $M$ as defined in Theorem~\ref{thm:poweroftwo}, $\calF_1(G)$ has at least $2^{M-1}$ special simplices.
\end{corollary}

\begin{proof}
For each ample framing of $G$, we get a unique special simplex.
Since there are $2^{M-1}$ ample framings, we have at least that many special simplices in the flow polytope.
\end{proof}


\bibliographystyle{plain}
\bibliography{bibliography}

\end{document}